\documentclass[preprint,12pt]{elsarticle}
\usepackage{latexsym}
\usepackage{amsmath}
\usepackage{graphicx}
\usepackage{float}
\usepackage{subcaption}
\usepackage{xcolor}
\usepackage{amssymb}

\DeclareGraphicsExtensions{.ps, .eps, .pdf, .png}
\newcommand{\ta}{\mathcal{H}_{B(z,r)}(x)}
\newcommand{\tb}{\mathcal{H}_{B(z,r)}}
\newcommand{\eeq}{\end{equation}}
\newcommand{\beq}{\begin{equation}}
\newcommand{\nuq}[1]{\label{#1} \eeq}

\newcommand{\h}{\mathcal{H}}
\newcommand{\mz}{\mu(B(z,r))}

\newtheorem{remark}{Remark}
\newtheorem{lemma}{Lemma}

\newtheorem{proposition}{Proposition}
\newenvironment{proof}[1][Proof]{\noindent\textbf{#1.} }{\ \rule{0.5em}{0.5em}}

\journal{Physica D}
\begin{document}
\begin{frontmatter}
\title{
Generalized dimensions, large deviations and the distribution of rare events}
\author{Th\'eophile Caby\fnref{labelgiorgio,lasa}}
\author{Davide Faranda\fnref{lapa1,lada}}
\author{Giorgio Mantica\fnref{labelgiorgio,labelg2,labelg3}}
\author{Sandro Vaienti\fnref{lasa}}
\author{Pascal Yiou\fnref{lapa1}}

\address[labelgiorgio]{Center for Nonlinear and Complex Systems, Dipartimento di Scienza ed Alta Tecnologia, Universit\`a degli Studi dell' Insubria, Como, Italy}
\address[lasa]{Aix Marseille Universit\'e, Universit\'e de Toulon, CNRS, CPT, 13009 Marseille, France}
\address[lapa1]{Laboratoire des Sciences du Climat et de l'Environnement, UMR 8212 CEA-CNRS-UVSQ, IPSL and Universit\'e Paris-Saclay, 91191 Gif-sur-Yvette, France}
\address[lada]{London Mathematical Laboratory, 8 Margravine Gardens, London, W6 8RH, UK}
\address[labelg2]{INFN sezione di Milano, Italy}
\address[labelg3]{Indam, Gruppo Nazionale di Fisica Matematica, Italy}

\begin{abstract}
Generalized dimensions of multifractal measures are usually seen as static objects, related to the scaling properties of suitable partition functions, or moments of measures of cells. When these measures are invariant for the flow of a chaotic dynamical system, generalized dimensions take on a dynamical meaning, as they provide the rate function for the large deviations of the first hitting time, which is the (average) time required to connect any two different regions in phase space. We prove this result rigorously under a set of stringent assumptions. As a consequence, the statistics of hitting times provides new algorithms for the computation of the spectrum of generalized dimensions. Numerical examples, presented along with the theory, suggest that the validity of this technique reaches far beyond the range covered by the theorem.

We state our result within the framework of extreme value theory. This approach reveals that hitting times are also linked to dynamical indicators such as stability of the motion and local dimensions of the invariant measure. This suggests that one can use local dynamical indicators from finite time series to gather information on the multifractal spectrum of generalized dimension. We show an application of this technique to experimental data from climate dynamics.

\end{abstract}


\end{frontmatter}

\section{Introduction and summary of the paper}
\label{sec1}

Generalized dimensions are a primary tool for the analysis of multifractal measures, {\em i.e.}, measures whose local densities feature a range of different scaling exponents.
Interest in these quantities originated in the eighties of the last century \cite{grass,grassp,multif1}, primarily for the study of chaotic attractors and fully developed turbulence \cite{uriel,benzi} and rapidly became important also from the mathematical viewpoint. The combined effort of physicists and mathematicians lead to the development of the so--called thermodynamical formalism \cite{rue,pesin-book,pesjstat} in which generalized dimensions play a major role.  For the sake of numerical experiments, but also of application to empirical data, many different techniques have been proposed along the years for the numerical calculation of generalized dimensions (we just quote \cite{grassp,holger,remo,antonio,rolf,rie,peinke} because even a partial list of references would require a full paper).

In this work we link generalized dimensions and the recurrence properties of the dynamics. In the same line of thought of Kac's theorem \cite{kac}, and more generally of ergodic theory, we study the connection between a {\em dynamical} quantity, the hitting time of a small set, and a {\em static} quantity, the statistical distribution of the measure of small balls.

This approach has been initiated in \cite{grass,gra2,natureh}:  generalized dimensions can be derived from the moments of the so--called first return time, the length of time required for the dynamics to return close to a chosen initial point on the attractor. Numerical experiments in \cite{noiprl} showed applicability and failures of this technique: in \cite{iojstat} the relation between dimensions generated via return times and the original quantities has been examined rigorously and in full generality, also providing explicit examples and counter--examples.

A related result will be derived herein, by considering hitting times \cite{CU}, rather than return times: hitting times are related to targeting small regions of phase space, starting from a different, arbitrary point. {\bf Proposition \ref{prop-1}} in Section \ref{sec-largedev} shows that dimensions indexed by $q \leq 2$ (the meaning of this index, which is related to the order of moments, will be explained in the next section) can be computed in this way. In Section \ref{sub-hit}, the abstract theory is translated into a numerical algorithm and applied to the test cases of the Arnol'd cat map and the H\'enon attractor.

A second fundamental point of this paper is to show that generalized dimensions yield the rate function for the large deviations of the first hitting time of a ball of given radius.
They quantify the rate at which the probability of observing ``non typical'' values diminishes when the radius goes to zero. Our result, {\bf Proposition \ref{prop-two}} in Section \ref{sec-largedev}, parallels a previous investigation \cite{CU} where the target set in phase space is a dynamical cylinder, rather than a ball. Rigorous theory is presented for the case of conformal repellers\footnote{These are the invariant sets of uniformly expanding $C^{1+\alpha}$ maps, defined on smooth manifolds and whose derivative is a scalar times an isometry. The repeller arises as the attractor of  pre--images of the map, see \cite{ba} for an exhaustive description. Dynamically generated Cantor sets on the line, Iterated Function Systems with the open set condition, disconnected hyperbolic Julia sets, are all examples \cite{dem} of conformal repellers. It is worth mentioning that such repellers can be coded by a subshift of a finite type and they support invariant measures which are Gibbs equilibrium states. This makes them particularly suited for the application of the thermodynamic formalism.} but we believe that similar results hold in more general settings.
A numerical illustration of the large deviation statistics is presented in Section \ref{sub-local}, in the case of a dynamical system evolving on  the Sierpinski gasket.

To obviate the limitation of the hitting technique to a part of the spectrum of generalized dimensions, in Section \ref{sec-gdevt} we continue the investigation of a preceding paper by some of the present authors \cite{D2}, in which the correlation dimension (the dimension of index $q=2$) was computed by means of extreme value theory (EVT).  The dynamical extremal index (DEI) appearing  in the Gumbel's limiting law followed as a by-product and it was interpreted as the rate of backward contraction on the unstable subspaces, a quantity closely related to positive Lyapunov exponents. We now extend this technique to the case of arbitrary, positive, integer index $q$.  In Sections \ref{sub-exce} and \ref{sub-block} the abstract theory is applied to specific examples.
The associated DEI is again related to Lyapunov exponents, but in Section \ref{sub-dei}, via {\bf Proposition \ref{prop-20}} and the subsequent analysis, we show that it is also affected by the variation of the invariant density and by the  lack of uniform hyperbolicity of the system.

We consider the rate function of the hitting time as a  way to detect and quantify the presence of  rare events in the dynamics. These events are produced by the presence of points where the local dimensions and the first hitting time (our statistical indicators) do not assume their typical values.
Although these non-typical points (like {\em e.g.} unstable fixed points, periodic points) have null probability to be attained by the dynamics, their influence in a finite region around their location affects the convergence of statistical indicators via an exponentially small probability of deviations from the typical values. The rate function measures the intensity of such deviations.

To show a realization of this scenario, in Section \ref{sec-climate} we analyze experimental data coming from climate dynamics. In fact, our broader goal is to implement statistical tools to investigate and to interpret data coming from various physical situations, like climate dynamics, but also turbulence, neuroscience and biology.  Further applications of the methods developed in this paper will appear in forthcoming publications.

\section{Definitions and review of related literature}

Let $(M,\mu,T)$ be a dynamical system given by a map $T$ acting on a metric space $M$ with distance $d(.,.)$ and preserving a Borel probability measure $\mu$. If we denote by $B(z,r)$ the ball of radius $r$ centered at $z \in M$, we define the spectrum $D_q$ ($q \neq 1$) of the generalized dimensions of $\mu$ by the scaling relation of the $q$--correlation integral with respect to the radius $r$

\beq
 \Gamma_\mu(r,q) =  \int_M \mu(B(z,r))^{q-1} d\mu \sim_{r \to 0} r^{D_q(q-1)}.
\nuq{1}
For $q=1$ the above is replaced by:
\beq
\int_M \log(\mu(B(z,r))) d\mu \sim_{r \to 0} {D_1} \log r.
\nuq{2}
See \cite{grass} and \cite{gra2} for the introduction of the theory, \cite{pesjstat} for a formal rigorous definition of the previous scaling behavior and the books \cite{holger, PP, ba} for several applications to  non-linear systems.

The scaling Eq. (\ref{1}) can be made mathematically precise by the following limit that defines the real function $\tau$ of the index $q$ (employing liminf and limsup to define upper and lower quantities, when they differ):
\begin{equation}\label{d2}
\tau(q)=\lim_{r\rightarrow 0}\frac{\log \Gamma_\mu(r,q)}{\log r}.
\end{equation}
Generalized dimensions $D_q$ are obtained from the function $\tau(q)$ via the equation
\beq
\tau(q)=D_q(q-1)
\nuq{eq-tau}
when $q \neq 1$, and by l'Hopital rule when $q=1$.
It is well known that, for a large class of dynamical systems, the Legendre transform of  $\tau(q)$, namely
\begin{equation}\label{FA}
f(\alpha)=\min_q\{\alpha\  q-\tau(q)\},
\end{equation}
is the Hausdorff dimension of the  set of points $z\in M$ verifying:
\begin{equation}\label{ED}
\lim_{r\to0} \frac{\log\mu(B(z,r))}{\log r}=\alpha,
\end{equation}
provided the limit exists, see \cite{pesjstat, ba} and references therein. This limit is called the {\em local dimension} of the measure $\mu$ at the point $z.$

The so-called {\em exact dimensional} measures $\mu$ have a local dimension that is constant $\mu-a.e.$. This dimension coincides with the information dimension $D_1$ \cite{LSY}. Several dynamical systems with hyperbolic properties possess an invariant measure that is exact dimensional, whose information dimension can be expressed in terms of the Lyapunov exponents and of the metric entropy. When the function $f(\alpha)$ is not singular, the space $M$ can be parted into uncountably many subsets characterized by the same local dimension, which are of zero measure (except of course that of dimension $D_1$) but of positive Hausdorff dimension, yielding what is called a multifractal.

When the generalized dimensions $D_q$ vary with $q$, they imply deviations of the local dimensions defined in Eq. (\ref{ED}) from the expected value $D_1$.  To gauge the deviation of the observable $\frac{\log\mu(B(z,r))}{\log r}$ from its expected value, at the finite resolution $r>0$, one considers the quantity $\mu\left(\left\{ z \in M \mbox{ s.t. } \frac{\log\mu(B(z,r))}{\log r}\in I\right\}\right),$ where $I$ is any interval in $\mathbb{R},$ including or not the expected value $D_1$. It has been recently proven \cite{SJ} the interesting result that in the family of systems known as conformal repellers the previous deviations decrease exponentially when $r$ tends to zero, with a rate that is given in terms of the generalized dimensions:
\beq
\lim_{r\to0}\frac{1}{{\log r}} \log \mu\left(\left\{z \in M \mbox{ s.t. } \frac{\log\mu(B(z,r))}{\log r}\in I\right\}\right)=\inf_{s\in I}Q(s).
\nuq{5ab}
The {\em rate function} $Q(s)$ is determined again by $\tau(q)$ of Eq. (\ref{d2}):
\beq
Q(s)=\sup_{q\in \mathbb{R}}\{-qs+\tau(q+1)\}.
\nuq{5b}

In addition to the point $z$, let us now consider a second point $x \in M$, and let us denote by $\h_{B(z,r)}(x)$ the first hitting time of the point $x$ in the ball $B(z,r)$:
\beq
\ta=\min\{n > 0 \mbox{ s.t. } T^n(x)\in B(z,r)\}.
\nuq{eq-hit1}
A particular situation happens when the point $x$ belongs to the ball $B(z,r)$. In this case one calls $\ta$ the time of first return of $x$ into $B(z,r)$. It is convenient to define $\mu_{|B(z,r)}(\cdot)$, the restriction of $\mu$ to $B(z,r)$:
\[
\mu_{|B(z,r)}(A) = \frac{\mu(A)}{\mu(B(z,r))},
\]
where $A$ is any measurable set.
When the invariant measure $\mu$ is ergodic, it is well known that the first return time satisfies Kac's theorem \cite{kac}:
\begin{equation}\label{kacb}
 \mathbb{E}_{\mu_{|B(z,r)}}(\tb)=
 \int_{B(z,r)} \tb (x) d\mu_{|B(z,r)}(x) =
 \frac{1}{\mu(B(z,r))}.
\end{equation}
By using the previous result in association with the Ornstein and Weiss theorem \cite{OW}, it is possible to show that the first return time satisfies
\beq
\lim_{r \to 0} \frac{\log{\cal H}_{B(z,r)}(z)}{-\log r}=D_1,
\nuq{OW}
for $z$  chosen $\mu-a.e.,$ (see \cite{STV1, STV2}). Observe that  in the above equation we are considering the first return of the center of the ball into the ball itself, {\em i.e.} $x=z$.
The first return time enjoys exponential large deviations, namely it was proven in  \cite{SJ} that:
\beq
\mu\left(\left\{z \in M \mbox{ s.t. }\frac{\log\tb(z)}{-\log r}\in I\right\}\right)\sim r^{\inf_{s\in I}Q^*(s)}.
\nuq{5a}
The rate function $Q^*$ is slightly different from the $Q$ given above.  For its precise definition we refer again to \cite{SJ} Theorem 2.5,  where the large deviation property is proven\footnote{Actually, even for conformal repellers, the limit in Eq. (\ref{5ab}) must be replaced with $\liminf$ and $\limsup$ and the rate functions are complicated expressions involving $\tau(q)$.}.

The question has been asked whether a multifractal description of the first return time could be meaningful, by considering the set of points where the limit in Eq. (\ref{OW}) is different from the typical value $D_1$. Yet, in the case of conformal repellers it has been proven  \cite{SS1, SS2}, that {\em all} level sets with a value different from $D_1$ have the same Hausdorff dimension of the ambient space, see also \cite{olsen}. This is a further point in favor of using large deviations instead of the multifractal description for studying recurrence quantities.

Let us now return in full generality to hitting times, when the initial condition $x$ does not necessarily belong to a neighborhood of the final state $z$. For systems with super-polynomial decay of correlations  a result analogous to Eq. (\ref{OW}) holds:
\beq
\lim_{r \to 0} \frac{\log\ta}{-\log r}=D_1,
\nuq{ga}
for $x$ and $z$ chosen $\mu-a.e.$ \cite{sg}. The next question to ask is whether hitting times enjoy exponentially large deviations, that is, whether and under what conditions it holds true that
\begin{equation}\label{LDP}
\mu\times\mu\left(\left\{(x,z) \in M \times M \mbox{ s.t. } \frac{\log\tb(x)}{-\log r}\in I\right\}\right)\sim r^{\inf_{s\in I}\hat{Q}(s)},
\end{equation}
where the rate function $\hat{Q}$ presumably involves again the generalized dimensions.  Notice that we weighted the event with the product measure since there are two sources of {\em alea}, in the choice of the starting point $x$ and of the target point $z$. Such large deviation property has interesting physical consequences.
As anticipated above, we expect that the presence of points $x$ and $z$ giving limits different from $D_1$ in Eq. (\ref{ga})---which we interpret respectively as exceptional initial conditions ($x$) and rare target regions ($z$)---yields deviations in the limit of Eq. (\ref{ga}) for small $r$. These deviations go to zero exponentially fast, in a well-defined limit procedure, with a rate which is measurable and that can be linked to the intensity of the extreme events.

In the next section we prove such results for a specific class of dynamical systems, linking them to generalized dimensions as in \cite{CU}. We will then provide two techniques to compute generalized dimensions for positive and negative $q$, in both cases by using a recurrence approach. In particular, for positive $q$ we  use extreme value theory and as a by product we  obtain a new sets of {\em extremal indices} that we interpret in terms of the Lyapunov exponents and of the density of the invariant measure, thus extending previous results in \cite{D2} for $q=2.$

\section{Large deviations for the first hitting times}
\label{sec-largedev}
Extreme value theory (EVT) can be used to determine the probability that the system enters for the first time a small region of the phase space (rare event) after a certain amount of time \cite{book}.
Instead of looking directly to the probability of the first occurrence of  such rare events, one could ask whether the presence of those events influence the convergence of the indicators  towards their expected values. This can be achieved by looking at the deviations from typical values  and the rate of such deviations can be obtained using EVT.

We  now state and prove a general result on large deviations in the statistics of the first hitting time. The result relies on a set of assumptions that hold true for several dynamical systems possessing some sort of hyperbolicity and exponential decay of correlations. We therefore consider dynamical systems $(M, \mu, T)$ that verify the following assumptions:

\begin{itemize}
\item {\bf A-1: Exponential distribution of hitting times with error}. There is a constant $C>0$ such that for $\mu$-a.e. $z\in M$ and $t>0$ we have
\begin{equation}\label{E2}
\left|\mu\left(\left\{x \in M \mbox{ s.t. } \ta\ge \frac{t}{\mu(B(z,r))}\right\}\right)-e^{-t}\right|\le C \delta_r \max(t,1)\ e^{-t}
\end{equation}
where
\beq \delta_r=O\left(\mz|\log\mz| \right).
\nuq{eq-a1a}
In particular, for $t>\mz |\log\mz^C|$ we have:
\begin{equation}\label{RE1}
\mu\left(\left\{x \in M \mbox{ s.t. } \ta\ge \frac{t}{\mu(B(z,r))}\right\}\right)=\exp[-t(1+O(\delta_r)](1+O(\eta_r)),
\end{equation}
with $\eta_r=O(\mz)$, while for $t\le \mz |\log\mz^C|$ we have\footnote{In the proof of Proposition \ref{prop-1} we set the constant exponent $C$ to unity because its value is irrelevant for the proof.}
\begin{equation}\label{RE2}
\mu\left(\left\{x \in M \mbox{ s.t. } \ta\ge \frac{t}{\mu(B(z,r))}\right\}\right)\ge 1-\frac{t}{C}.
\end{equation}
Notice that the above implies that both $\eta_r$ and $\delta_r$ depend on the point $z \in M$.
\item {\bf A-2: Exact dimensionality.} The measure $\mu$ verifies Eq. (\ref{ED}) and the limit value is $\alpha = D_1$ for $\mu$-a.e. $z$.
    \item{\bf A-3: Uniform bound for the local measure.}
        There exists $d^*>0$ such that for all $z\in M$ we have
        \begin{equation}\label{UB}
        \mu(B(z,r))\le r^{d^*}.
        \end{equation}
\item {\bf A-4: Existence and analyticity of the correlation integrals.} For all $q\in \mathbb{R}$ the limit defining $\tau(q)$, Eq. (\ref{d2}), exists.
Moreover the function $\tau(q)$ is real
analytic for all $q\in \mathbb{R}$, $\tau(0)=-D_H,$ $\tau(1)=0,$ $\tau'(q)\ge 0$
and $\tau''(q)\le 0.$ In particular $\tau''(q)<0$ if and only if $\mu$ is not a measure of maximal entropy.
\end{itemize}

We derive the first assumption from Keller's paper \cite{GK}, where condition (\ref{E2}) is proven for the so-called REPFO maps\footnote{REPFO stands for {\em Rare events Perron-Frobenius operators}, since the conditions are given in terms of the spectral properties of the transfer operator.},
which include a large class of mixing systems with exponential decay of correlations. Keller's derivation of  (\ref{E2}) contains fine estimations of the quantity $\mu\left\{ \ta\ge \frac{t}{\mu(B(z,r))}\right\}$ for $t\sim \mz |\log \mz|$, which we adopted in formulae (\ref{RE1}) and (\ref{RE2}). Moreover,  Keller's conditions are even more general, since they hold for any point $z$, provided the rescaled time $\frac{t}{\mu(B(z,r))}$ is modified as $\frac{t}{\kappa_r\mu(B(z,r))}$, where the factor $\kappa_r$ depends on the target point $z$ and converges to the {\em extremal index} at $z$ when $r$ goes to zero. For the kind of ``nice'' expanding systems we are considering, including the REPFO ones, this extremal index is equal to one almost everywhere, see also \cite{book} for an extensive discussion, and this explains our
choice in eq. (\ref{RE2}). Assumption A-4 is a strong one and it has been proven to hold for conformal mixing repellers endowed with Gibbs measures in \cite{pesjstat}. This condition has also been assumed in \cite{SJ}, to prove the large deviation result (\ref{5a}). For the same class of conformal repellers Assumption A-3 holds too, see Lemma 3.15 in \cite{SJ}. As remarked in the Introduction, we consider these {\em ideal} systems interesting models to establish rigorous results that might also hold in more general settings.
\begin{proposition}
\label{prop-1}
Let us suppose that the dynamical system $(M,\mu,T)$  verifies Assumptions (A-1)-(A-4). Then:
\begin{itemize}
\item For $q>0$,
\beq
 \lim_{r\rightarrow 0}\frac1{\log r}\log \int_M \int_M\ta^{q-1} d\mu(x) d\mu(z)=\lim_{r\rightarrow 0}\frac1{\log r}\log \int_M \mz^{1-q} d\mu(z).
\nuq{eq-prop1a}
\item For $q\le 0$,
\beq
\lim_{r\rightarrow 0}\frac{1}{\log r}\log \int_M \int_M \ta^{q-1} d\mu(z)d\mu(x)=\lim_{r\rightarrow 0}\frac{1}{\log r}\log \int_M \mz d\mu(z).
\nuq{eq-prop1b}
\end{itemize}
\end{proposition}
\begin{proof}
We follow the scheme of the proof of Theorem 3.1 in \cite{CU}, by translating its argument from cylinders to balls. We use a simple lemma whose proof is a standard exercise:
\begin{lemma}
Consider a function $f$ from $M$ to the integer numbers larger than, or equal to one: $f : M \rightarrow \mathbb{N}_+$. Let $0<A\leq 1$ and define
\beq
I(q) = \int_M f^q (x) d \mu(x).
\nuq{eq-add}
Then, when $q>0$,
\beq
I(q) = 1 - \lim_{t \to \infty} t^q \mu\left(\left\{x \in M \mbox{ s.t. } f(x) > t \right\}\right) + \frac{q}{A^q} \int_A^\infty t^{q-1} \mu\left(\left\{x \in M \mbox{ s.t. } f(x) > \frac{t}{A} \right\}\right) dt.
\nuq{eq-lem1}
On the other hand, when $q<0$,
\beq
I(q) = - \frac{q}{A^{q}} \int_A^\infty t^{q-1} \mu\left(\left\{x \in M \mbox{ s.t. } f(x) < \frac{t}{A} \right\}\right) dt.
\nuq{eq-lem1b}
\end{lemma}
To prove Proposition \ref{prop-1} we need to apply the previous lemma to $f(x) = \ta$ and consider the integral
\beq
I(q-1,z,r) =  \int_M\ta^{q-1} d\mu(x),
\nuq{eq-hiti2}
which is of the form studied in the Lemma.
There are three cases to consider.

\begin{enumerate}
\item Case 1: $q>1$, that is, $q-1>0$ that allows to apply formula (\ref{eq-lem1}).
Because of assumption {\bf A-1}, Eq. (\ref{E2}), the limit in Eq. (\ref{eq-lem1}) is null. Moreover, using again Eq. (\ref{E2}) for $t>1$,
\beq
e^{-t}(1-C\delta_rt )\le \mu\left(\left\{x \in M \mbox{ s.t. } \ta\ge\frac{t}{\mz}\right\}\right)  \le e^{-t}(1+C\delta_rt),
\nuq{eq-dif1}
we can bound the integrand at r.h.s. in Eq. (\ref{eq-lem1}):
\[
I(q-1,z,r) \geq
 \frac{q-1}{\mu(B(z,r))^{q-1}}
\left[\int_{1}^\infty t^{q-2}  e^{-t} dt - C\delta_r \int_{0}^\infty t^{q-1}  e^{-t} dt \right].
\]
The last two integrals are convergent. Moreover, observe that Eq. (\ref{eq-a1a}) and {\bf A-3} imply that $\delta_r$ is uniformly bounded from above. Taking $r$ small enough the term into brackets becomes positive and larger than a quantity independent of $r$.  A similar reasoning yields an upper bound for $I(q-1,z,r)$ of the same form. Since the double integral $\int_M \int_M\ta^{q-1} d\mu(x) d\mu(z)$ is the single integral of $I(q-1,z,r)$ with respect to $d\mu(z)$, the equality (\ref{eq-prop1a}) follows.

\item Case 2: $0<q<1$. Since $q-1 < 0$ we employ Eq. (\ref{eq-lem1b}):
\beq
I(q-1,z,r) =  \frac{1-q}{\mu(B(z,r))^{q-1}} \int_{\mu(B(z,r))}^\infty
t^{q-2}
\mu\left(\left\{x \in M \mbox{ s.t. } \ta < \frac{t}{\mz}\right\}\right) dt.
\nuq{eq-qmin1}
We use again Eq. (\ref{eq-dif1}) to get
\[
I(q-1,z,r) \geq \frac{1-q}{\mu(B(z,r))^{q-1}} \int_{1}^\infty
t^{q-2} (1-e^{-t} (1+C \delta_r t)) dt =   \frac{1-q}{\mu(B(z,r))^{q-1}} S(z,r).
\]
In the above we have put
\[
S(z,r) = \int_{1}^\infty
t^{q-2} (1-e^{-t}) dt +  C \delta_r \int_{1}^\infty t^{q-1} e^{-t} dt,
\]
The term $S(z,r)$ is again composed of a constant (the first convergent integral) and of a vanishing quantity (when $r \to 0$), which puts us in the position of using the previous technique to obtain a first inequality between the two terms of Eq. (\ref{eq-prop1a}).

To prove the reverse inequality we begin by observing that for $t<1$ condition {\bf A-1} simply becomes $|\mu\left(\left\{x \in M \mbox{ s.t. } \ta\ge \frac{t}{\mu(B(z,r))}\right\}\right)-e^{-t}|\le O(\delta_r);$ then
we start from Eq. (\ref{eq-qmin1}) and we part the integral in two, the first from $\mz$ to one and the second from one to infinity:
\beq
I(q-1,z,r) = \frac{1-q}{\mu(B(z,r))^{q-1}} (J_1 + J_2).
\nuq{eq-cacchio3}
We use again Eq. (\ref{eq-dif1}) to get:
\beq
   J_1 = \int_{\mz}^1 t^{q-2}
\mu\left(\left\{x \in M \mbox{ s.t. } \ta < \frac{t}{\mz}\right\}\right) dt \leq
\nuq{eq-j1a}
\beq
\int_{\mz}^1 t^{q-2} (1-e^{-t}+O(\delta_r )) dt \leq \int_{0}^1 t^{q-2} (1-e^{-t}) dt + O(\delta_r )\int_{\mz}^1  t^{q-2}  dt.
\nuq{eq-j1a2}
The first integral in the above is a positive constant. Let us consider the second term:
\[
O(\delta_r) \int_{\mz}^1  t^{q-2}  dt
\leq \frac{O(\delta_r)}{1-q}  \mz^{q-1} = O\left(\mz^q |\log\mz \right)|,
\]
where the last equality follows from $\delta_r=O\left(\mz |\log\mz \right)|$, Eq. (\ref{eq-a1a}). Therefore, when $r$ tends to zero, this term vanishes.

The case of $J_2$ is easier:
\beq
   J_2 = \int_{1}^{\infty} t^{q-2}
\mu\left(\left\{x \in M \mbox{ s.t. } \ta < \frac{t}{\mz}\right\}\right) dt \leq
\int_{1}^{\infty} t^{q-2} dt = C_2,
\nuq{eq-j2a}
which is again bounded by a constant.

\color{black}
\item Case 3: $q<0$. We use again Eq. (\ref{eq-cacchio3}), with $J_1$ and $J_2$ defined in Eqs. (\ref{eq-j1a}) and (\ref{eq-j2a}), respectively. For the latter integral the inequality (\ref{eq-j2a}) still holds, with a different constant $C_2$. To deal with the integral $J_1$ we further split its domain into the intervals $[{\mz},{\mz |\log\mz|}]$ and $[{\mz |\log\mz|},1]$, thereby defining the integrals $J_{1,1}$ and $J_{1,2}$, respectively.

At this point we use Eq. (\ref{RE2}) to estimate  from above the integrand of $J_{1,1}$ and Eq. (\ref{RE1}) to do the same for $J_{1,2}$. Putting the two estimates together, we obtain
 \begin{equation}
 J_1 \le C_3 \int_{\mz}^{1}t^{q-1}dt\le \frac{C_3}{|q|}
\left[\mz^{q}- 1 \right]
\le   \frac{C_3}{|q|}
\mz^{q},
 \end{equation}
where $C_3$ is another constant independent of $r$ and $z$.

To get a lower bound we write
\[
J_1 = \int_{\mz}^{1}t^{q-2}\mu\left(\tb<\frac{t}{\mz}\right)dt\ge
 \int_{\mz}^{1}t^{q-2}\mu(\tb\le1)\; dt
\]
\[
= \mz\int_{\mz}^{1}t^{q-2}dt = |q-1|^{-1}\mz^{q}\left[1-\mz^{|q-1|}\right]
\]
\[
\ge |q-1|^{-1}\mz^{q}\left[1-r^{d^*|q-1|}\right].
\]
In the last step we have used Eq. (\ref{UB}), so that the term in the square brackets is positive and uniformly bounded for $r$ small enough.
\end{enumerate}
By collecting all the
preceding estimates, we get the desired result for all $q$. \end{proof}

We are now ready to state our result on large deviations of the first
hitting time. We first recall that the  {\em free energy function} $R(q), q\in \mathbb{R}$ associated with the process $\frac{\log\ta}{-\log r},$  is given by
\beq
R(q)=\lim_{r\rightarrow 0}\frac{1}{-\log r}\log \int_M\int_M  \ta^{q} d\mu(z)d\mu(x),
\nuq{9}
provided the limit exists. If $R(q)$ is $C^2$ and strictly convex on $\mathbb{R},$
its Legendre transform is called
the rate function $\hat{Q}$ and  satisfies $\hat{Q}(s)=\sup_q\{qs-R(q)\}$; we refer to \cite{SJ}
for a brief review of large deviations, see also \cite{EE}. Our previous Proposition shows that the free energy for the first hitting time verifies $R(q)=-\tau(1-q)$ when $q>-1.$ In this range of values of $q$, $R'(q)=\tau'(1-q)>0$ (by Assumption 2) and therefore the
supremum for the rate function $\hat{Q}$ is attained for positive $s$ by
a value of $q$ satisfying $R'(q)=s$. On the other hand, Assumptions {\bf A-2} and {\bf A-4} immediately imply that for positive $s$, $\hat{Q}(s)$ is a smooth convex
function with the minimum at $D_1.$ Since the free energy is not smooth everywhere, being  not differentiable in $q=-1,$ we cannot use the standard G\"artner-Ellis theorem, but a local version of it, as it is reported in Lemma XIII.2 in \cite{HH}. Let us put $\Delta=D_1+1,$ the above proves the crucial
\begin{proposition}
\label{prop-two}
Let us suppose that $\mu$ is not a measure of maximal dimension, which
ensures that $R(q)$ is strictly convex. Then for all $s\in (0, R(\Delta)/\Delta)$ we have
\beq
\lim_{r\rightarrow 0}\frac{1}{\log r}\log (\mu \times \mu)\left\{\frac{\log\ta}{-\log r}>D_1+s\right\}=\hat{Q}(D_1+s),
\nuq{largedev}
where $\hat{Q}(s)=\sup_q\{qs+\tau(1-q)\}.$
\end{proposition}
\begin{remark}
It is interesting to observe that the Legendre-Fenchel transform of the free energy  function $R(q)$ introduced above allows us to get the rate functions of the large deviations of different processes, namely:
\begin{itemize}
\item $R(q)$ gives the rate function $Q(s)$ of the information dimension,  see (\ref{5ab}).
\item $R(-q)$ gives the rate function $\hat{Q}(s)$ of the first hitting time, see (\ref{largedev}).
\item $R(q-1)$ gives the function $f(\alpha)$ expressing the Hausdorff dimension of the level sets with local dimension $\alpha$, see(\ref{FA}).
\end{itemize}
We therefore consider $R(q)$ an important global tool to analyze and describe the geometric and recursive properties of dynamically invariant measures.
\end{remark}

\section{Numerical determination of generalized dimensions}
The numerical determination of generalized dimensions is a principal concern when experimental data are to be examined, or theoretical hypotheses need to be tested on model cases. Many techniques have been proposed \cite{remo,rolf,rie,peinke} especially to deal with the case of negative dimensions ({\em i.e.} those corresponding to a negative value of $q$) that call in cause rarified regions of the invariant measure. Via Kac's theorem, these latter are related to large return times, hence rare events. For this reason a return time approach seems particularly suited to treat this case.

\subsection{Hitting time integral}
\label{sub-hit}
In this section we follow this approach, based upon Proposition \ref{prop-1}.
We assume that the data at our disposal are finite trajectories of the dynamical system $(M,\mu,T)$, which we label as $x_j = T^j (x_0)$, where the point $x_0$ is to be chosen on the attractor of the dynamical system. We set as reference technique the evaluation of the correlation integral in Eq. (\ref{1}) via a Birkhoff summation. This is effected by first finding the Birkhoff estimate of the measure of a ball
$ \mu(B(z_0,r))
$
via
\begin{equation}\label{kac2f}
J_N(z_0,r) = \frac{1}{N}\sum_{j=0}^{N-1} \chi_{B(z_0,r)}(x_j) \simeq  \mu(B(z_0,r)).
\end{equation}
Here and in the following $\chi_A$ is the indicator function of the set $A$.
The above quantity is then raised to the power $q-1$, and a second average with respect to the point $z_0$ is performed:
\begin{equation}\label{kac2g}
\frac{1}{N'}\sum_{l=0}^{N'-1} [J_N(T^l(z_0),r)]^{q-1} \simeq \Gamma_\mu(r,q).
\end{equation}
While in principle ($x_0$, $N$) and ($z_0$, $N'$) can be different, one can take advantage of the choice $z_0=x_0$, so to use a single trajectory for the computation. We elect not to do so, for reasons that we will explain later. A complete discussion of the method of correlation integrals is presented in \cite{holger}.

Once the correlation integral has been estimated (remark that the above Eqs. (\ref{kac2f}) and (\ref{kac2g}) are not scaling relations, but estimates that can be made arbitrarily precise), the problem remains of finding the scaling exponents implied in Eq. (\ref{1}) or equally the function $\tau(q)$.
Two main ways exist to do this from data at finite resolution $r$. The first is to employ extrapolation techniques of the ratio $\log \Gamma_\mu(r,q) / \log r$ computed at a number of values $r$, such as Levin's algorithm or similar \cite{roberto}. This is particularly useful when the measure has a hierarchical structure. The second, more conventional, is to try to find a linear least square fit of the log of the correlation integral with respect to the log of $r$. It is immediate to see that this is a sort of l'Hopital's rule to find the limit for $r$ tending to zero in Eq. (\ref{d2}).

In the present context, Proposition \ref{prop-1} implies that, for $q \leq 2$, the function $\tau(q)$ can be equally seen as the limit of the ratio in the left hand side of Eq.  (\ref{eq-prop1a}), after the substitution $q \to 2-q$. It requires the computation of the hitting double-integral
\beq
\Upsilon_\mu(q,r) = \int_M \int_M\ta^{1-q} d\mu(x) d\mu(z).
\nuq{eq-hiti1}
A Birkhoff estimate of this quantity can be obtained as follows: first consider the inner integral in the above equation (it was defined as $I(1-q,z,r)$ in Eq. (\ref{eq-hiti2}) of the previous section). This can be estimated as
\beq
 I_N(1-q,z_0,r) = \frac{1}{N} \sum_{j=0}^{N-1} \mathcal{H}_{B(z_0,r)}^{1-q}(x_j)\simeq  I(1-q,z_0,r).
\nuq{eq-hiti3}
In practice, it is convenient to fix $N$ (and to stop the evaluation of the motion) as soon as the trajectory of $x_0$ has entered the ball $B(z_0,r)$ $H$ times. In so doing, $N$ becomes a function of $H,$ $z_0$  and $x_0$. This can be done also when evaluating the conventional correlation integral, Eqs. (\ref{kac2f}) and (\ref{kac2g}). Next, we estimate the outer integral in (\ref{eq-hiti1}), again by a Birkhoff summation
\beq
 \frac{1}{N'} \sum_{l=0}^{N'-1} I_N(1-q,T^l(z_0),r)\simeq \Upsilon_\mu(q,r).
\nuq{eq-hiti4}
This procedure has the advantage that {\em the same set of data} can be used to determine an approximation to both the correlation integral $\Gamma_\mu(r,q)$ and the hitting integral $\Upsilon_\mu(q,r)$, so that the two methods can be compared fairly.

As first example of this comparison we choose the Arnol'd cat map on the two--torus \cite{cat}, a primary example of chaotic dynamical system, with the absolutely continuous invariant measure $\mu$ given by the Lebesgue uniform measure on this manifold, so that all generalized dimensions of the measure are equal to two. In Figure \ref{fig-1}, left panel, we plot the numerically estimated integrals $\Gamma_\mu(r,q)$ and $\Upsilon_\mu(q,r)$ versus $r$ in double logarithmic scale, for a selected set of values of $q$ ranging from $q=-1$ to $q=2$ and of $r$ ranging from $r=10^{-3}$ to $r=10^{-1}$. The (trivially constant) data for $q=1$ separate the integrals $\Gamma_\mu(r,q)$, $\Upsilon_\mu(r,q)$ that grow from those that diminish when $r$ tends to zero. The almost linear shape of the curves confirms the scaling in Eq. (\ref{1}) and a linear fit as described above provides an estimate of $\tau(q)$ and hence of $D_q$. Yet, a finer analysis reveals that the asymptotic behavior is not yet achieved at finite $r$.  In fact, in the right panel the results obtained using the slope of each linear interpolation between successive values of $r$ in the figure are displayed: since values of $r$ are equally spaced in logarithmic scale, we define
\beq
  \sigma_q(r) = \frac{1}{q-1} \frac{ \log \Upsilon_\mu(q,\rho r) - \log \Upsilon_\mu(q,r)}
    {\log(\rho r) - \log {r}},
    \nuq{eq-hiti5}
with $\rho < 1$. The values obtained are not constant: the lowest set of data, in particular, is related to $q=2$, the highest value for which generalized dimensions can be obtained in this way. Its value for $r = 1.77 \; 10^{-3}$ is still far from the theoretical value $D_2=2$, even if $\sigma_q(r)$ is an acceleration procedure of the limit in Eq. (\ref{d2}). Nonetheless, a further extrapolation can be performed. Typically, convergence in these estimates is rather slow, in the sense that a behavior of the kind
 \beq
  \sigma_q(r) = D_q + B  \log (r)
  \nuq{eq-hiti6}
holds. Therefore, using $D_q$ and $B$ as fitting parameters of the experimental data, a better estimate of $D_q$ can be obtained. The continuous line in the right panel of Figure \ref{fig-1} plots such approximation. The obtained value of $D_2$ is correct to three digits.
Finally, for comparison, the data obtained by using $\Gamma$ {\em in lieu} of $\Upsilon$---in other words, the conventional correlation integral---are also reported, shifted upwards by a small quantity for clarity. Recall that they have been computed on the same raw data (trajectories) than the former. They show a reduced precision for small values of $r$, in correspondence with the {\em smallest} value in the range plotted in figure, $q=-1$.

This instability is shown in Figure \ref{fig-sigma2c}, which reports the same data of the right panel of Figure \ref{fig-1}, for $q=-1$. In the same figure the analogue data obtained by using the {\em first--return time} integral \cite{noiprl,iojstat}
\beq
\Gamma_\tau(q,r) = \int_M \mathcal{H}_{B(x,r)}^{1-q}(x) \; d\mu(x)
\nuq{eq-hiti10}
are also reported. They are evaluated using a small portion of the data used in the other two cases: in fact, in this case only the {\em first} return is concerned, rather than the $H$ hits required by the previous techniques. Despite the fact that a rigorous proof of this procedure is lacking (see nonetheless \cite{iojstat}) data are consistent with the expected result.

\begin{figure}[h!]
    \centering
    \begin{subfigure}[t]{0.5\textwidth}
        \centering
\includegraphics[height=2.5in,angle=270]{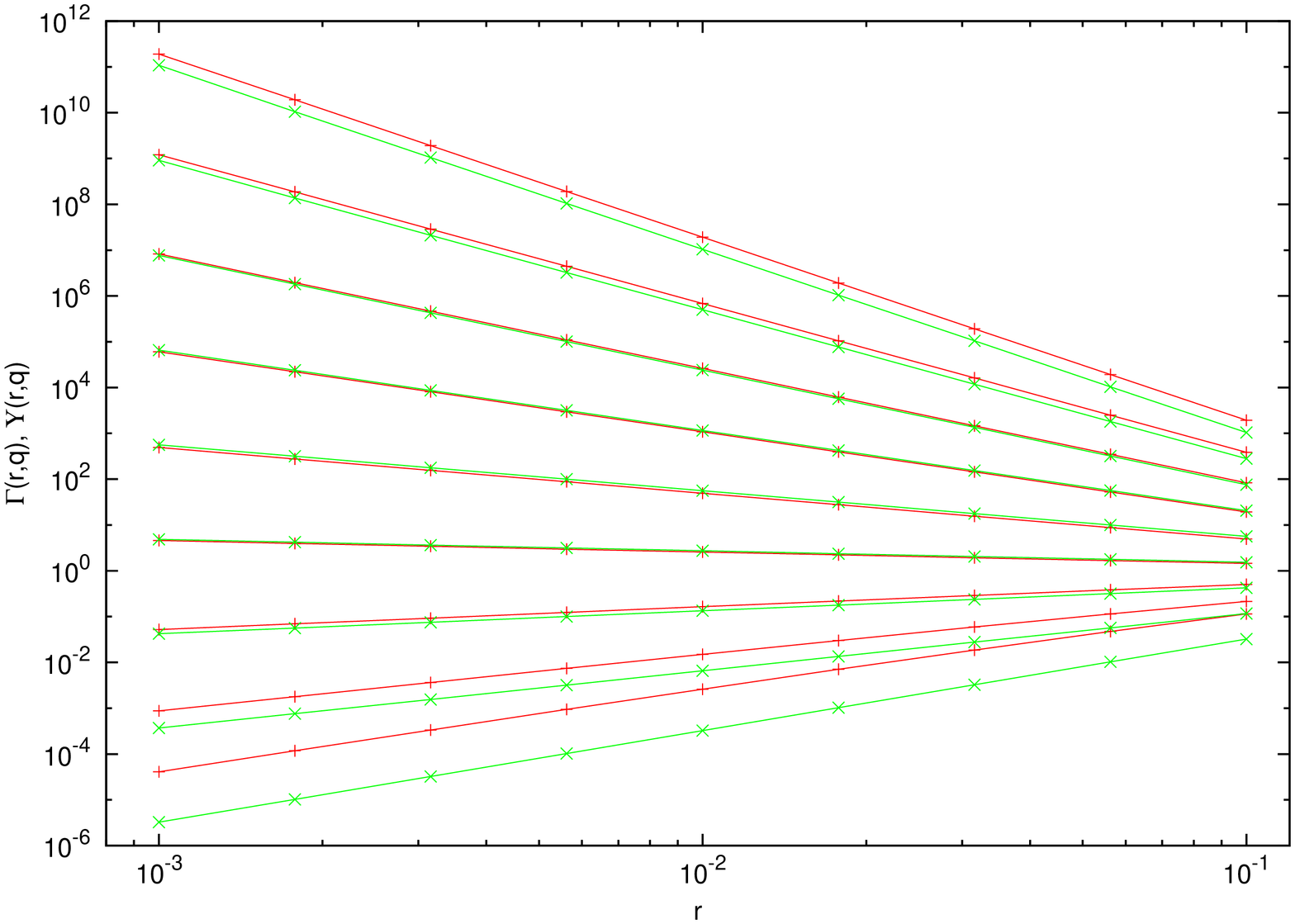}
    \end{subfigure}%
    ~
    \begin{subfigure}[t]{0.5\textwidth}
        \centering
      \includegraphics[height=2.5in,angle=270]{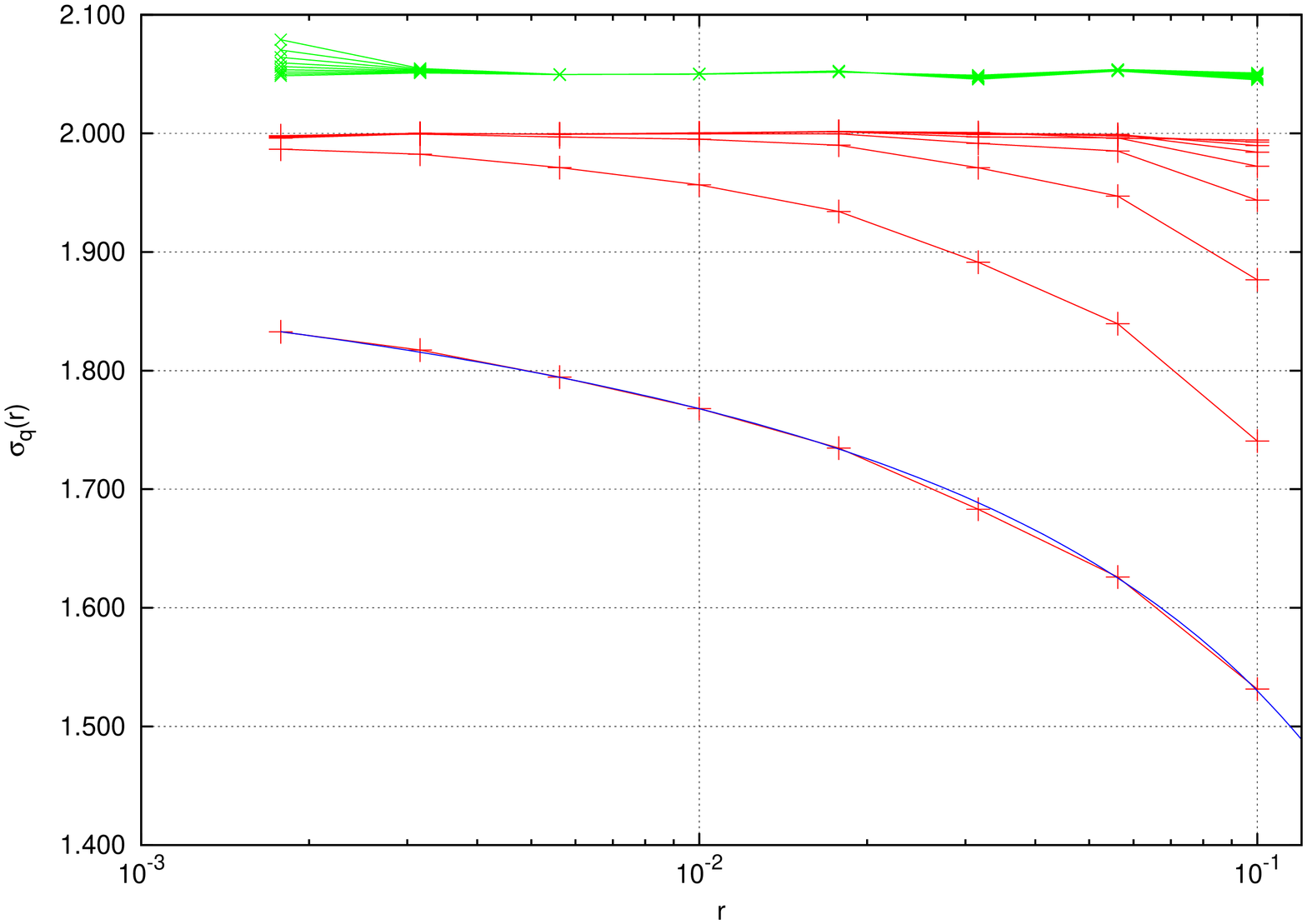}
    \end{subfigure}
   \caption{Left panel: correlation integral $\Gamma_\mu(r,q)$ (green lines) and hitting integral $\Upsilon_\mu(q,r)$ (red lines) evaluated numerically by the procedure of eqs. (\ref{kac2f}) -- (\ref{kac2g}) and (\ref{eq-hiti3}) -- (\ref{eq-hiti4}) with $H=32$, $N'=256,000$. Lines join values with the same $q$, ranging from $q=-1$ (highest curve) to $q=2$ (lowest). Right panel: slopes $\sigma_q(r)$ extracted from $\Upsilon_\mu(q,r)$ in the left panel, following Eq. (\ref{eq-hiti5}) (red). Values of $q$ range from $q=-1$ (highest curve) to $q=2$ (lowest). The blue curve is the fit given by Eq. (\ref{eq-hiti6}) with $D_2= 2.006$ and $B= 1.095$. Finally, values of $\sigma_q(r)$ extracted from $\Gamma_\mu(q,r)$ are plotted in green, shifted upwards by $.05$.
   }
   \label{fig-1}
\end{figure}

\begin{figure}
\centerline{\includegraphics[height=10cm,angle=270]{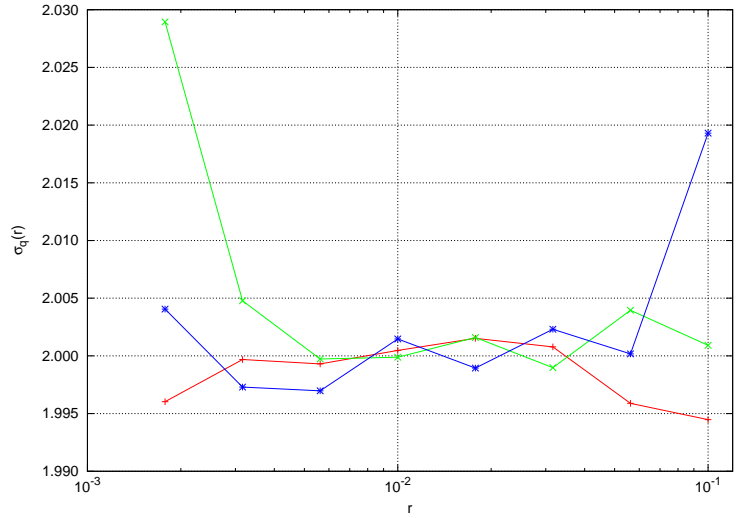}}
\caption{As in Figure \ref{fig-1}: data for $q=-1$. Also plotted (blue) are the data obtained by the first return time integral $\Gamma_\tau(q,r)$, Eq. (\ref{eq-hiti10}).}
\label{fig-sigma2c}
\end{figure}

A second example, when the hypotheses of Proposition \ref{prop-1} are certainly {\em not} verified, is given by the H\'enon map, at standard parameter values \cite{hen}.
Figure \ref{fig-hen0} is the analogue of Fig. \ref{fig-1} in this second case, with the only difference that in the right panel a three dimensional figure displays the quantity $\sigma_q(r)$ versus $r$ and $q$, computed from the correlation integral $\Gamma_\mu(q,r)$ and the hitting time integral $\Upsilon_\mu(q,r)$. In both cases we observe the large local fluctuations typical of the H\'enon physical measure. More importantly, a considerable agreement between the two sets of data is observed, when $q$ is smaller than two.

In the successive Figure \ref{fig-hen2} the generalized dimension obtained by linear least square fit over the full range of the data in Figure \ref{fig-hen0}, left panel, are displayed. It is well known that, in the case of the physical measure on the H\'enon attractor, generalized dimensions strongly depend on the range of the fit --- as well as on the sampling point chosen in this range. We do not aim to resolve this issue, but we remark that the coincidence between the results obtained by the correlation integral and the hitting time integral suggests that the results of Proposition \ref{prop-1} hold also in this case, which is clearly outside the scope of the hypotheses put forward in the previous section.

Finally, still in Figure \ref{fig-hen2}, we also plot the curve $D_2/(q-1)$, which follows from Proposition \ref{prop-1} and describes the scaling of the hitting time integral for $q$ larger than or equal to two. As in the case of Arnol'd cat, data for $q$ approaching two from below are not at convergence, while those for $q$ significantly larger than two fit the theoretical curve remarkably well.

\begin{figure}[h!]
    \centering
    \begin{subfigure}[t]{0.5\textwidth}
        \centering
    \includegraphics[height=2.5in,angle=270]{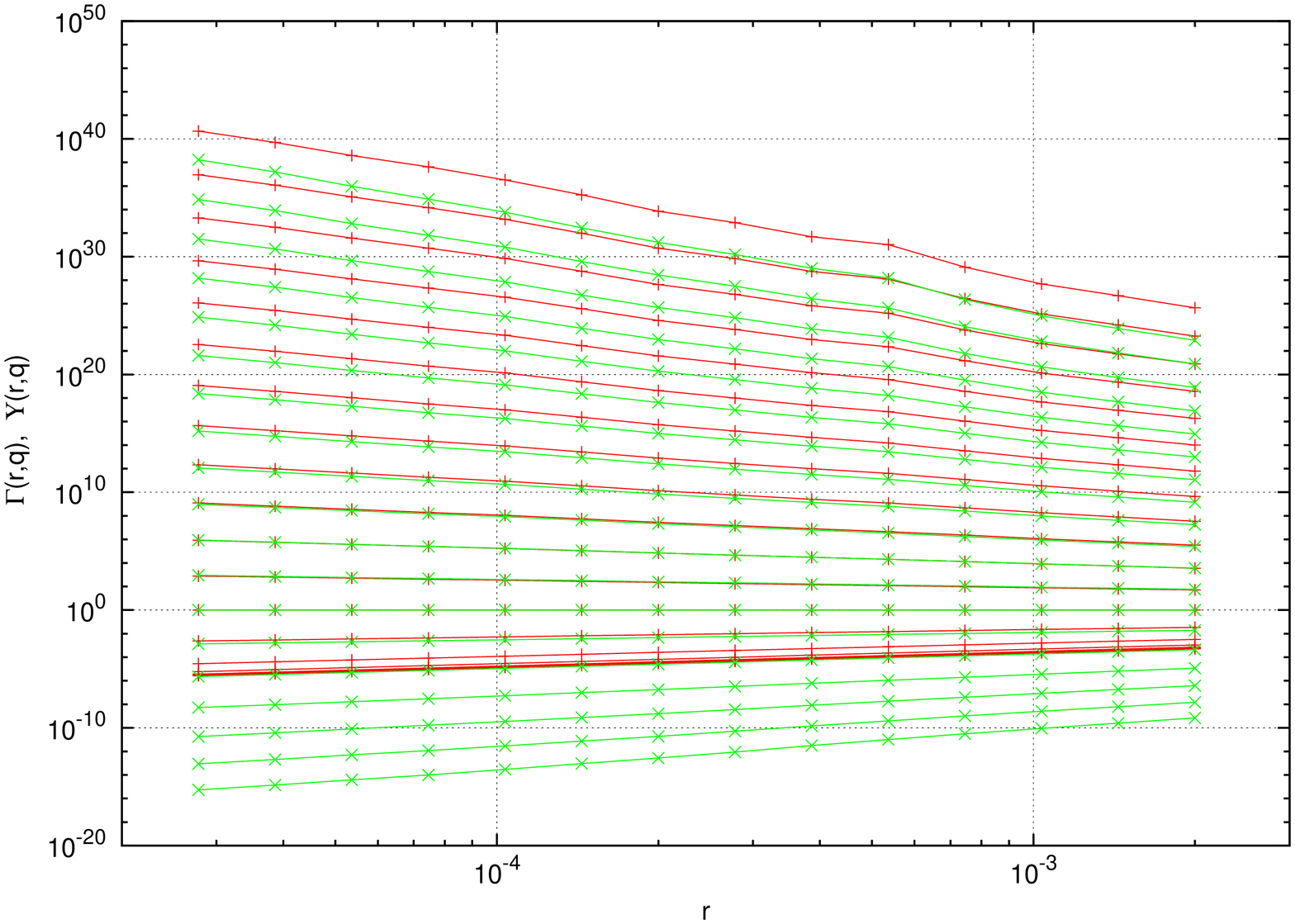}
    \end{subfigure}%
    ~
    \begin{subfigure}[t]{0.5\textwidth}
        \centering
       \includegraphics[height=3in,angle=270]{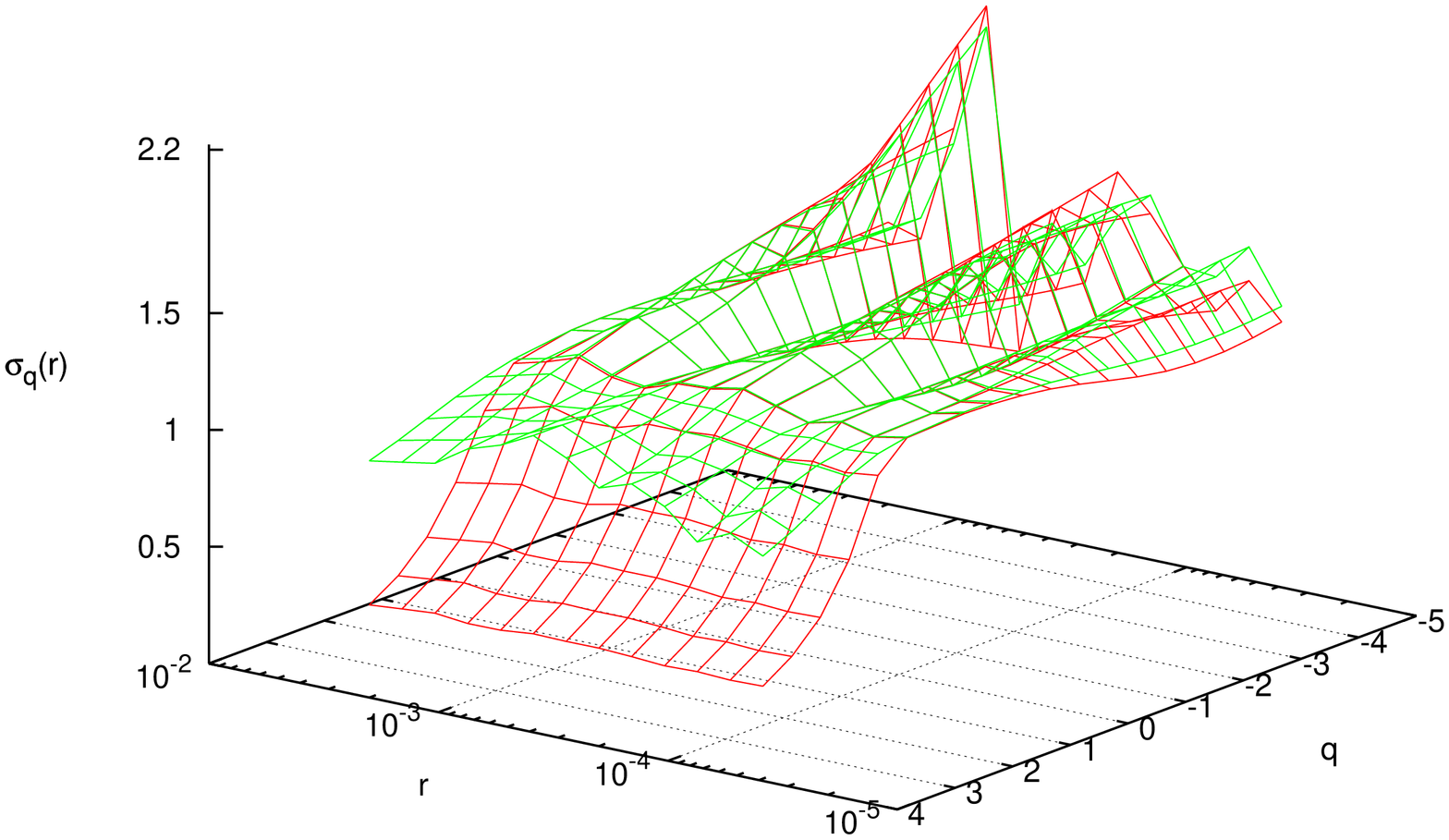}
    \end{subfigure}
   \caption{Left panel: correlation integral $\Gamma_\mu(r,q)$ (green lines) and hitting integral $\Upsilon_\mu(q,r)$ (red lines) evaluated numerically by the procedure of eqs. (\ref{kac2f}) -- (\ref{kac2g}) and (\ref{eq-hiti3}) -- (\ref{eq-hiti4}) with $H=64$, $N'=256,000$. Lines join values with the same $q$, ranging from $q=-5$ (highest curve) to $q=4$ (lowest). Right panel: slopes $\sigma_q(r)$ versus $r$ and $q$ extracted from
   $\Gamma_\mu(q,r)$ (green) and $\Upsilon_\mu(q,r)$ (red) in the left panel, following Eq. (\ref{eq-hiti5}).
   }
   \label{fig-hen0}
\end{figure}

\begin{figure}
\centerline{\includegraphics[height=10cm,angle=270]{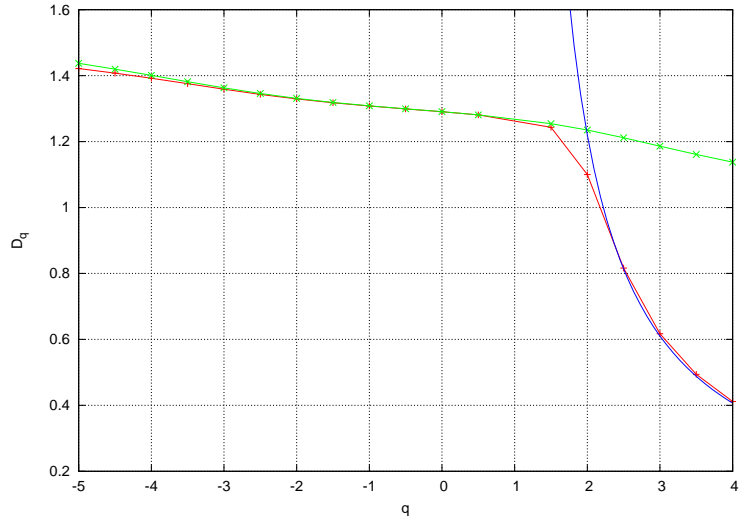}}
\caption{Generalized dimensions obtained by fitting the data in Fig. \ref{fig-hen0}, left panel. Dimensions obtained from the correlation integral are plotted in green, from the hitting time integral in red. Plotted in blue is the curve $D_2/(q-1)$, implied by Proposition \ref{prop-1}.}
\label{fig-hen2}
\end{figure}

These experimental data leads us to conclude that the theoretical method to determine generalized dimensions implied by Proposition \ref{prop-1} has a practical value, but, for values of $q$ between one and two, convergence must be accelerated by suitable techniques.
We finally remark that, as conjectured in \cite{noiprl,iojstat}, the same can be expected when using first return times.

\subsection{Using local dimensions computed via EVT}
\label{sub-local}

As described above, the key to the computation of generalized dimensions is the estimate of the measure of balls of the same radius $r$, raised to a power and averaged with respect to the invariant measure. While dimensions are obtained via a scaling relation of these quantities when the radius vanishes, the distribution of such measures at fixed, finite $r$, is also important. This observation leads to the definition of a {\em finite resolution} local dimension, $D_{1,r}(z)$, which is precisely defined by the equation:
\beq
\mu(B(z,r)) = r^{D_{1,r}(z)}.
\nuq{eq-locdim1}
It is interesting to note the relations of this quantity to extreme value theory. In fact, defining as observable the function
\beq
\phi_z(x)=-\log d(x,z),
\nuq{7}
where $z$ is the center of the ball in Eq. (\ref{eq-locdim1}), and computing this latter on a trajectory of the system, $x_j = T^j (x_0)$, large values of $\phi_z(x_i)$ correspond to passes of the motion close to the point $z$. By looking at the statistics of these extreme events--near approaches, one defines the (complementary) distribution function
\beq
    \bar{F_z}(u) = \mu(\{{x} \in M \mbox{ s.t. } \phi_z ({x}) >u \}),
\nuq{eq-phiz1}
which coincides with the measure of the ball of radius $e^{-u}$ around the point $z$. From the numerical point of view, as described in \cite[][Chapters 4 and 6 and references therein]{book}, one studies the tail of this distribution, either defined by considering arguments larger than $u_{\mbox{\tiny cut}} = - \log r_{\mbox{\tiny cut}}$, or by setting a cutoff value in the distribution itself: $\bar{F_z} < 1 - p$, with $p$ close to one. This second case yields a cutoff value $r_{\mbox{\tiny cut}}$, which now depends on position. Extreme value theory predicts that the tail distribution, suitably renormalized and shifted, converges for small cutoff to an exponential distribution, whose mean and standard deviation are the inverse of $D_{1,r_{\mbox{\tiny cut}}}(z)$. The latter is then numerically computed as the inverse of the mean of such distribution.  In other words, this is an alternative procedure to Eq. (\ref{kac2f}) that can be used in two ways.

Firstly, it can be turned into a determination of generalized dimensions.  Eq. (\ref{kac2g}) is here replaced by
\beq
\Gamma_\mu(r,q) \simeq \frac{1}{N} \sum_{j=0}^{N-1}r^{(q-1) D_{1,r}(T^jx)},
\nuq{8}
where $x$ is chosen $\mu-a.e.$ , $N$ is supposed to be large and dimensions are obtained by Eqs. (\ref{d2}) and (\ref{eq-tau}). It is sometimes both necessary and convenient {\em not} to take the limit of vanishing $r$ in Eq. (\ref{d2}). In practical applications this is sometimes dictated by the finite resolution of the data and the limited span of time evolution at our disposal. In Section \ref{sec-climate} this situation is illustrated by applying the above procedure to the spectrum of dimensions for the North Atlantic atmospheric circulation.

Secondly, the same computations permit to evaluate in a direct way the large deviation function of local dimension: by taking $I=(s,\infty)$ with $s>D_1$ in Eq. (\ref{5ab}), we find that $\mu(D_{1,r}>s) \sim r^{Q(s)}$. Similarly, we have that $\mu(D_{1,r}<s) \sim  r^{Q(s)}$ for $s<D_1$. Figure \ref{fig-serp} shows the numerically computed rate function $Q(s)$ for the motion on a Sierpinski gasket defined in Section \ref{sub-block}, Eq. (\ref{17.5}). By lowering the cutoff value of $r$, approach to the theoretical curve is observed. This theoretical value is given by the rate function $Q(s)$, which is computed as the Legendre-Frenchel transform of the free energy $R(q)=-\tau(1-q)$.  For the case of motion on the Sierpinski gasket, the function $\tau(q)$ is explicitly given by formula (\ref{TT}). Numerically, it can be obtained by the techniques described in this article, yielding results for $Q(s)$ more reliable than those obtained by the direct computation of the distribution of $\mu(D_{1,r}<s)$: this is undoubtedly an interesting result with potential applications to a wide class of dynamical systems.

\begin{figure}
\centerline{\includegraphics[height=10cm]{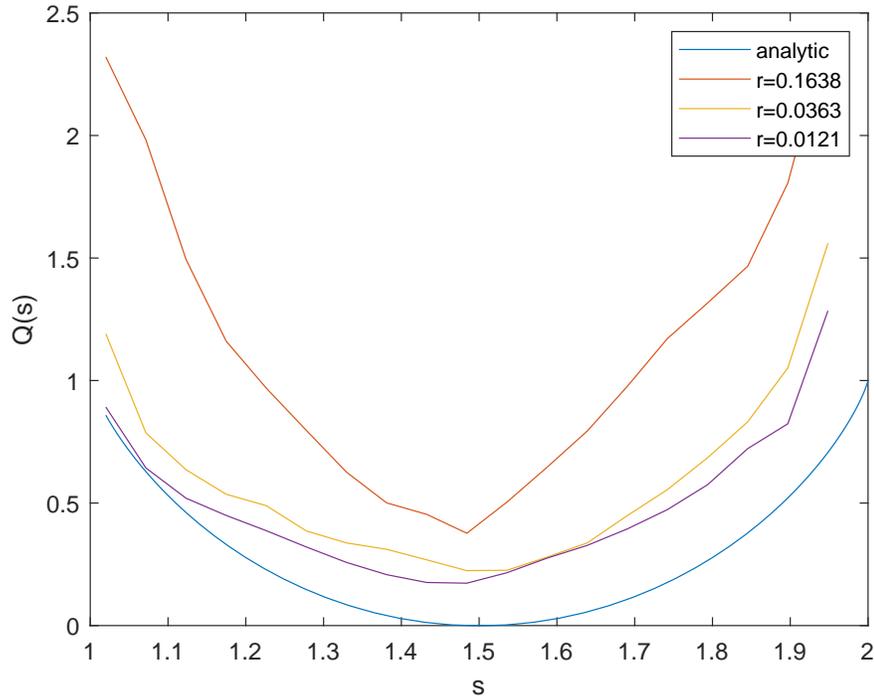}}
\caption{Rate function $Q(s)$ for the motion on a Sierpinski gasket, computed from $1,000$ sampling points, each of which required a trajectory consisting of $10^6$ iterates.}
\label{fig-serp}
\end{figure}

\section{Generalized dimensions via extreme value theory}
\label{sec-gdevt}
In this section we describe a further method to compute the spectrum of the generalized dimensions for positive, integer values of $q$ larger than one.  It has the advantage of using EVT intrinsically and, in addition, it reveals a second spectrum of extremal indices that also possesses a dynamical meaning. This approach is a direct generalization of the method introduced in \cite{D2} for the correlation dimension. It is based on the investigation of close encounters, when two or more trajectories of the system approach each other within a small distance. This defines the extreme event that we investigate.

Let us consider the $q-$fold ($q>1$) direct product $(M,\mu,T)^{\otimes q}$ with the direct product map $T_q=T\otimes \cdots \otimes T$ acting on the product space $M^q$ and the product measure $\mu_q=\mu \otimes \cdots \otimes \mu$. Define the following observable on $M^q$:
\beq
\phi(x_1,x_2,\dots,x_q)=-\log(\max_{i=2,\dots,q}d(x_1,x_i)),
\nuq{9b}
where each $x_i \in M$. We also write $\overline{x}_q=(x_1,x_2,\dots,x_q)$ and $T_q(\overline{x}_q)=(Tx_1,\dots,Tx_q)$.

\subsection{Statistics of exceedances}
\label{sub-exce}
Let us first investigate the statistical distribution of the function $\phi$, via the (complementary) distribution function $\bar{F}(u)$:
\beq
    \bar{F}(u) = \mu_q(\{\overline{x}_q \in M^q \mbox{ s.t. } \phi (\overline{x}_q) >u \}).
\nuq{eq-phiq1}
It is easily seen that
\beq
\bar{F}(u) =\int_{M^q} d \mu_{q}(\overline{x}_q) \chi_{B(x_1,e^{-u})}(\overline{x}_q)
\cdots \chi_{B(x_1,e^{-u})}(\overline{x}_q)
                     =\int_M d \mu(x_1) \mu (B(x_1,e^{-u}))^{q-1}.
\nuq{14a}
Comparing eq. (\ref{14a}) with eq. (\ref{1}) yields
$
 \bar{F}(u) \sim e^{-u_nD_q(q-1)},
$ 
so that one can obtain $\tau(q)=(q-1)D_q$ from the asymptotic behavior of $\bar{F}(u)$ for large $u$. This quantity can be estimated by a Birkhoff sum, involving the trajectories of $q$ different initial conditions of the original system. The results of this procedure in the case of the Arnol'd cat dynamical system are reported in Figure \ref{fig-manyp1}, in simple logarithmic scale, for values of $q$ ranging from $q=2$ to $q=8$. The linear parts of these graphs follow closely the theoretical result $\bar{F}(u) =\mu_q(\phi>u) =  \pi^{q-1} u^{2(q-1)}$. The limitations of the procedure are evident from the picture: for large values of $q$ multiple ``encounters'' become scarcer and scarcer, so that the linear part, from which generalized dimensions can be extracted by linear fitting, becomes increasingly narrow when the length of the sampling trajectory is finite. Observe that the number of iterations considered in our numerical simulation largely exceed those typically available in real--world applications. On the other hand, this technique is not directly affected by the {\em curse of dimensionality} which plagues box counting procedures (but not the correlation integral, or hitting/return times methods, for that matter).

To complete the analysis of the previous section we also examine the case of the H\'enon physical measure. Results are reported in Figure \ref{fig-manyq1}, in full analogy with Figure \ref{fig-manyp1}. The exponential decay is evident also here, and the slopes of the curves, together with eq. (\ref{15}), permit to extract the data $\tau(2)=1.2$, $\tau(3)=2.32$, $\tau(4)=3.3$, which imply the generalized dimensions $D_2=1.2$, $D_3=1.16$ and $D_4=1.1$. These values compare favorably with the extensive calculations in \cite{arneodo87}. Although the exponential decay of the data for larger values of $q$ is also evident, the data do not allow to estimate the associated dimensions with the same precision.

\begin{figure}
\centerline{
  \includegraphics[height=10cm,angle=270]{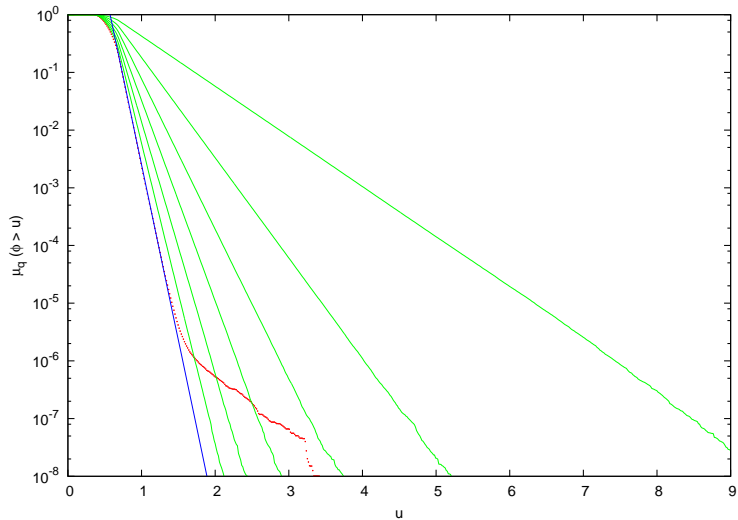}
}
\caption{Probability of large events, $\bar{F}(u) = \mu_q(\phi>u)$, versus $u$, in the case of the Arnol'd cat. It has been estimated via a Birkhoff average over 32 trajectories of length $10^{10}$. Data for $q=2$ (highest curve) to $q=7$ (green) and $q=8$ (red, lowest) have been reported. The theoretical result for $q=8$ is $\mu_q(\phi>u) = \pi^7 u^{14}$ (blue).}
\label{fig-manyp1}
\end{figure}
\begin{figure}
\centerline{
  \includegraphics[height=10cm,angle=270]{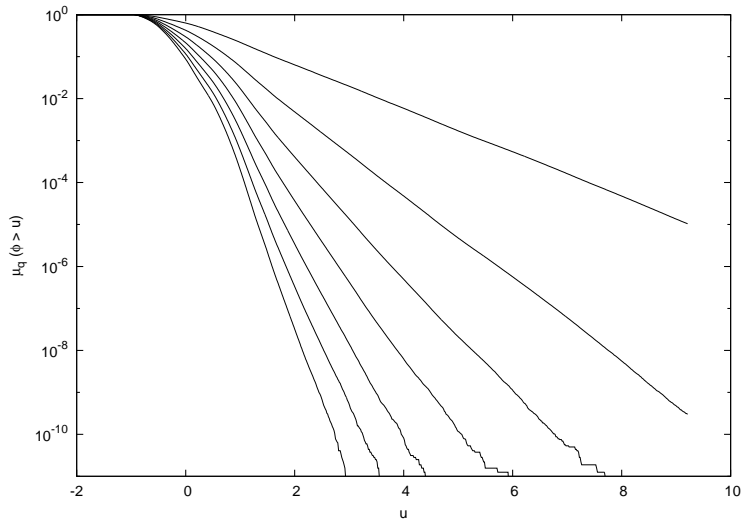}
}
\caption{Probability of large events, $\bar{F}(u) = \mu_q(\phi>u)$, versus $u$, in the case of the H\'enon attractor physical measure. It has been estimated via a Birkhoff average over 32 trajectories of length $10^{10}$. Data for $q=2$ (highest curve) to $q=8$ (lowest) have been reported. }
\label{fig-manyq1}
\end{figure}

\subsection{Statistics of block maxima}
\label{sub-block}
Let us now move more deeply into extreme value theory. It is a standard procedure, employed in the present context also in \cite{D2}, to consider the maximum value attained by the function $\phi$ over a {\em block of times} of length $n$. That is, we define the new observable
\beq
M_n(\phi;\overline{x}_q)=\max \{\phi(\overline{x}_q),\dots,\phi(T^{n-1}_q(\overline{x}_q)\},
\nuq{10}
and its distribution function ${F}_n(u)$:
\beq
    {F}_n(u) = \mu_q(\{\overline{x}_q \in M^q \mbox{ s.t. } M_n(\phi;\overline{x}_q) \leq u  \}).
\nuq{eq-phiq2}
Next, let $u_n$ be a sequence of real values which diverges at infinity, for which
\beq
 \bar{F}(u_n) \sim \frac{t}{n}
\nuq{13}
as $n$ tends to infinity and where $\bar{F}$ has been defined in Eq. (\ref{eq-phiq1}). In these equations, $t$ is a positive number (see Chapter 3 in \cite{book} for a general introduction to extreme value theory).
Under the hypotheses put forward in Section \ref{sec-largedev}, by using the spectral technique described in \cite{D2}, it is possible to prove the convergence of the distribution ${F}_n$, suitably rescaled, to the Gumbel's law
\beq
G(t) = e^{-\theta_q t}.
\nuq{12}
The quantity $\theta_q$ is called the dynamical extremal index DEI and it will be studied in the next section.

This convergence can also be investigated numerically and it provides an estimate of the generalized dimensions. In fact, because of eq. (\ref{13})
\beq
 \bar{F}(u_n) \sim e^{-u_nD_q(q-1)} \sim \frac{t}{n},
\nuq{15}
and
\beq
u_n \sim \frac{-\log t}{D_q(q-1)}+\frac{\log n}{D_q(q-1)}=\frac{-\log t}{a_n} + b_n.
\nuq{16}
The real quantities $a_n$ and $b_n$ can be obtained by  a maximum likelihood estimation of the GEV parameters in $F_n$ \cite{coles}. This is achieved numerically with the Matlab  {\em gevfit} function \cite{gevfit}. This was described in Section II-A of \cite{D2}, which yields the generalized dimensions $D_q$.

We apply this procedure to the case of an I.F.S. measure \cite{dem} on the Sierpinski gasket, defined by the stochastic process on the unit square $M=[0,1]^2$ realized by iteration of the maps $f_i$, $i=1,2,3$ chosen at random with probability $p_i$:
\beq
 \left\{
    \begin{array}{ll}
   f_1(x,y)=(x/2,(y+1)/2), \;p_1=1/4,\\
f_2(x,y)=((x+1)/2,(y+1)/2),\; p_2=1/4,\\
f_3(x,y)=(x/2,y/2),\; p_3=1/2.\\
    \end{array}
\right.
\nuq{17.5}
The distribution $F_n$ with block size $n=5000$ is estimated for 20 trajectories of $2\cdot 10^8$ points. In Figure \ref{fig-lor}, left panel, the numerically obtained generalized dimensions are compared with the analytical values \cite{rie}:
\begin{equation}\label{TT}
D_q=\frac{\log_2(p_1^q+p_2^q+p_3^q)}{1-q}.
\end{equation}
Good agreement is found for small values of $q$, which later worsens as expected and discussed earlier. In the same Figure \ref{fig-lor}, right panel, we also plot the results for the case of the Lorenz 1963 model \cite{lorenz},
a continuous--time dynamical system. Here, the distribution $F_n$ (with $n=10^4$) is obtained from trajectories of $10^8$ points, simulated by the Euler method (which is clearly not the best technique, but the focus of our investigation is different) with step size $0.013$. Dimensions estimates are obtained by averages over 20 trajectories, uncertainties being the standard deviations of these results.

\begin{figure}[h!]
    \centering
    \begin{subfigure}[t]{0.5\textwidth}
        \centering
        \includegraphics[height=2.in]{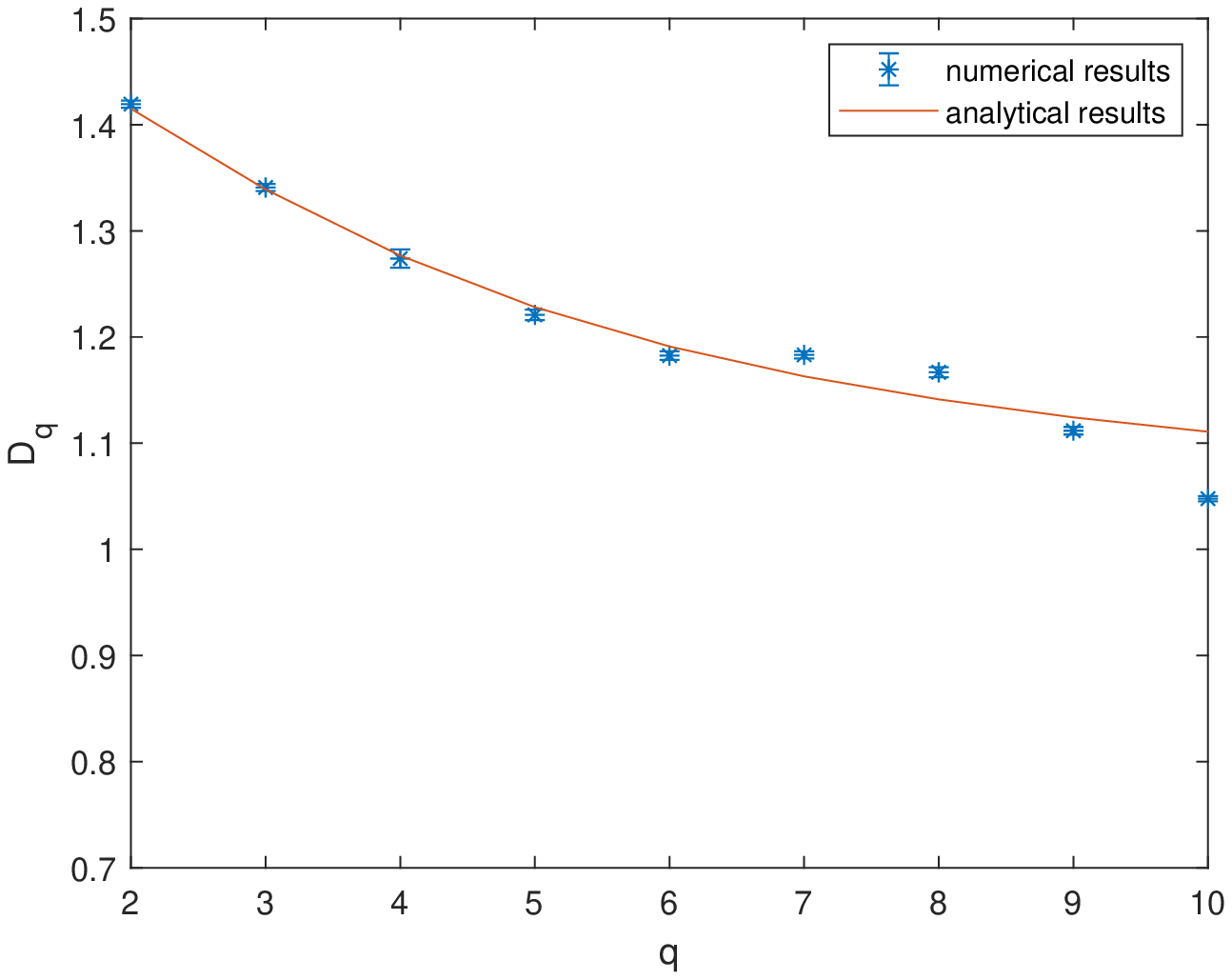}
    \end{subfigure}%
    ~
    \begin{subfigure}[t]{0.5\textwidth}
        \centering
     \includegraphics[height=2in]{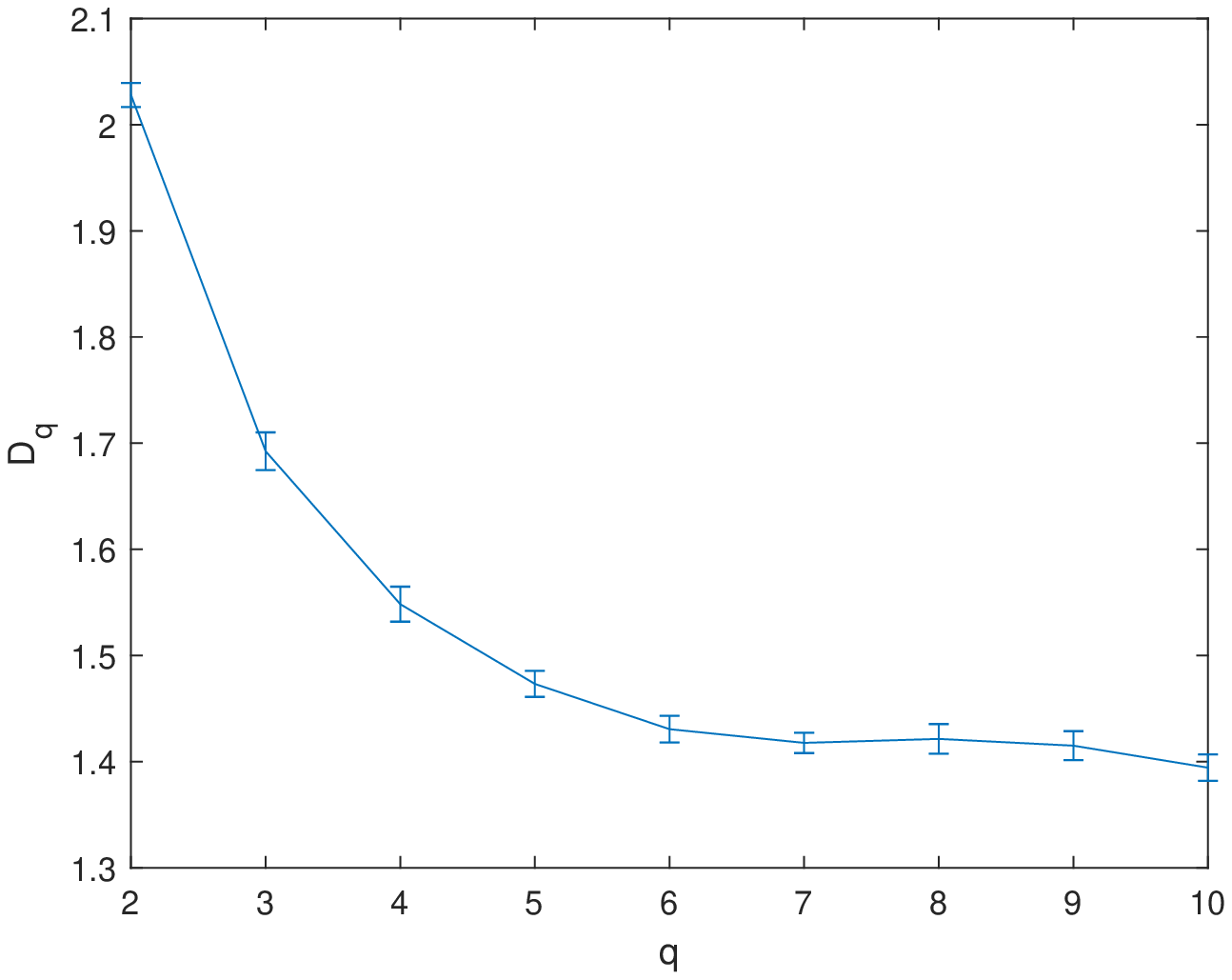}
    \end{subfigure}
   \caption{Left panel: numerical estimates of $D_q$ for the Sierpinski gasket (blue symbols) and theoretical value (red curve).  Right panel: numerical estimates of $D_q$ for the Lorenz 1963 model. The uncertainty is the standard deviation of the results obtained over 20 trajectories. See text for parameters and discussion.
   }
   \label{fig-lor}
\end{figure}

\section{The dynamical extremal index}
\label{sub-dei}

It is now important to consider the parameter $\theta_q$ appearing in the exponent of the Gumbel's law (\ref{12}); in \cite{D2} it was called the Dynamical Extremal Index (DEI). To do this, define the following subset of $M^q$:
\beq
\Delta_n^q=\{\overline{x}_q \in M^q \mbox{ s.t. }  d(x_1,x_2)<e^{-u_n}, \ldots ,d(x_1,x_q)<e^{-u_n}\}.
\nuq{17}
As argued in \cite{D2}, also using the spectral technique, based upon the analytical results in \cite{D1}, it is possible to show that:
\beq
\theta_q=1-\lim_{n \to \infty} \frac{\mu_q(\Delta_n^q \cap T_q^{-1}\Delta_n^q)}{\mu_q(\Delta_n^q)}.
\nuq{18}
For $C^2$ expanding maps of the interval, which preserve an absolutely continuous invariant measure $\mu=hdx$ with strictly positive density $h$ of bounded variation, it is possible to compute the right hand side  of (\ref{18}) and get:
\begin{multline}\label{20a}
\mu_q(\Delta_n^q \cap T_q^{-1}\Delta_n^q)=
\int dx_1h(x_1) \int dx_2h(x_2)\ \chi_{B(x_1,e^{-u_n)})}(x_2)\chi_{B(Tx_1,e^{-u_n)})}(Tx_2) \cdots \\
\cdots \int dx_q h(x_q)\chi_{B(x_1,e^{-u_n)})}(x_q)\chi_{B(Tx_1,e^{-u_n)})}(Tx_q).
\end{multline}
Each of the $q-1$ integrals above factorize, and they depend on the parameter $x_1$. Therefore they can be treated as in the proof of Proposition 5.5  in \cite{D1}, yielding the rigorous result:
\begin{proposition}\label{prop-20}
Suppose that: the map $T$ belongs to $C^2$; it preserves an absolutely continuous invariant measure $\mu=hdx$, with strictly positive density $h$ of bounded variation; it verifies conditions $P1-P5$ and $P8$ in \cite{D1}\footnote{These conditions essentially ensure that the transfer operator associated with the map $T$ has a spectral gap and that the density $h$ has finite oscillation in the neighborhood of the diagonal.}. Then
\beq
\theta_q=1-\frac{\int\frac{h(x)^q}{\mid DT(x)\mid^{q-1}}dx}{\int h(x)^qdx}.
\nuq{20}
\end{proposition}

This formula uses the translational invariance of the Lebesgue measure: we refer to sections II-B and II-C in \cite{D2} for analogous extensions to more general invariant measures and to SRB measures for attractors. As remarked in \cite{D2}, whenever the density does not vary too much, or alternatively the derivative (or the determinant of the Jacobian in higher dimensions) are almost constant, we expect a scaling of the kind:
\beq
\theta_q \sim 1-e^{-(q-1)h_m},
\nuq{21}
where $h_m$ is the metric entropy (the sum of the positive Lyapunov exponents).
This can be verified for the map $x \mapsto 3x \mbox{ mod }1$, for which eq. (\ref{20}) can be easily computed, giving $\theta_q=1-3^{1-q}$. Figure \ref{fig-3x} reports the numerical results. Extremal indices are computed using the S\"uveges estimator \cite{suveges} . For each $q$, we used as a threshold the $0.997$ quantile of the distribution ${F}(u)$ computed on a pre-runned trajectory of $10^6$ points. Results compare favorably with theory for $\theta_q$, but less satisfactorily for $H_q$ when $q$ becomes large. In the last case, it also turns out that the fixed threshold corresponds to lower values of the function $\phi$. Less precise results are also probably due to the fact that $\log(1-\theta_q)$ diverges as $q$ tends to infinity.

\begin{figure}[h!]
  \centering
  \begin{subfigure}[t]{0.5\textwidth}
       \centering
       \includegraphics[height=2in]{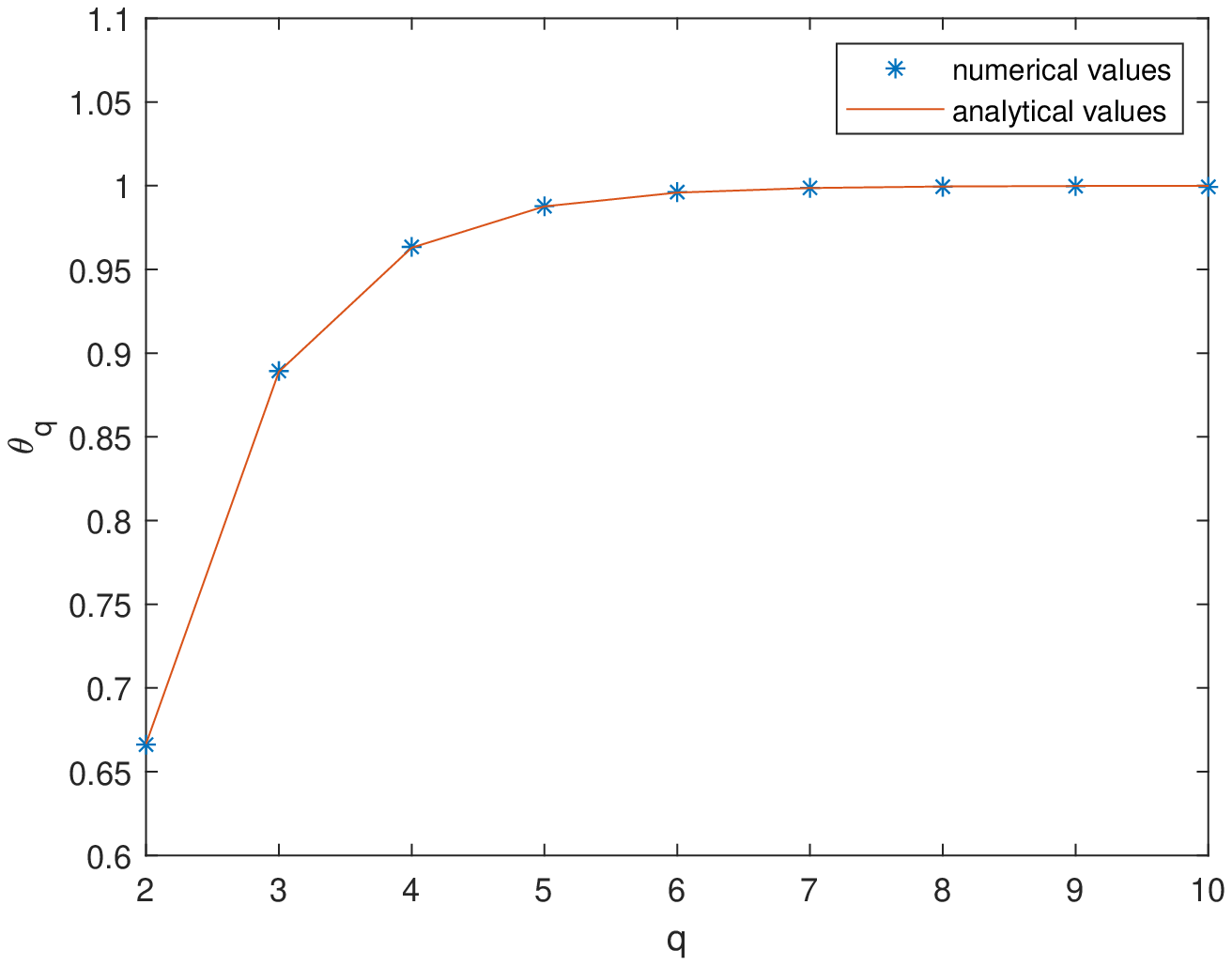}
        \caption{$\theta_q$}
    \end{subfigure}%
    ~
    \begin{subfigure}[t]{0.5\textwidth}
        \centering
       \includegraphics[height=2in]{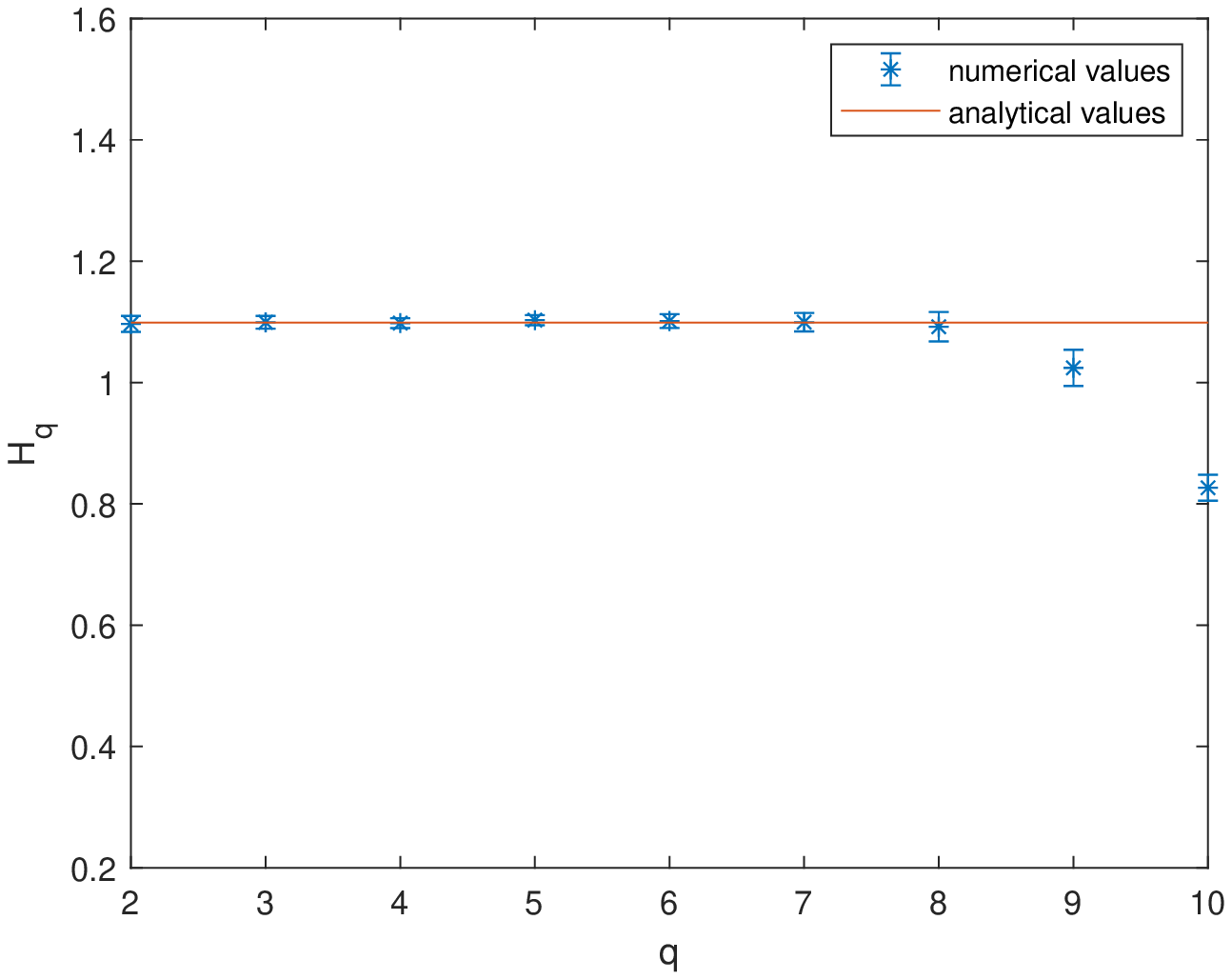}
        \caption{$H_q$}
    \end{subfigure}

    \caption{Indicators $\theta_q$ and $H_q$ of the map $x \mapsto 3x \mbox{ mod } 1$, obtained from averaging the results over 20 trajectories of $2\cdot 10^7$ points. The uncertainty is the standard deviation of the results. See text for further discussion.}
 \label{fig-3x}
 \end{figure}

Whenever the density, or the derivative, or both, exhibit appreciable variations, we expect a deviation from the scaling in eq. (\ref{21}). Hence, the variation with $q$ of the quantity
\beq
H_q=\frac{\log(1-\theta_q)}{q-1}
\nuq{24}
reveals how far we are from the positive Lyapunov exponent (or the entropy in higher dimensions). In particular we expect that such deviations will be magnified when:
\begin{itemize}
\item  the density has a minimum at zero, a fact that usually happens when the derivative blows up to infinity, or when the derivative vanishes somewhere, like in  multimodal maps,
    \item the density is unbounded, which happens when the derivative  has one as an eigenvalue on  periodic points, a typical occurrence for intermittent maps.
        \end{itemize}
In both cases it is not possible to bound between finite quantities the second term at right hand side of $\theta_q$ in Proposition \ref{prop-20}; more importantly, it is not at all clear that eq. (\ref{20}) still holds, since it was proved under the assumptions of a bounded and strictly positive density.

As an illustration of this theory, we now treat a few examples. In the numerical simulations, the extreme value distribution has been obtained by Birkhoff sampling of a trajectory of $N= 2 \cdot 10^7$ points. To compute the distribution of maxima, we used the peak over threshold approach \cite{book}, the threshold being the quantile $\sigma=0.995$ of the distribution.

\begin{itemize}
\item {\em Markov maps.} We consider the following piecewise linear Markov  map $T$ \cite{GB}:
\[
T(x) = \left\{
    \begin{array}{ll}
        T_1(x)= 3x \ \text{if}\  x \in I_1=[0,1/3),\\
        T_2(x)=5/3-2x \ \text{if} \   x\in I_2=[1/3,2/3),\\
        T_3(x)=-2+3x \ \text{if} \ x \in I_3=[2/3,1).
    \end{array}
\right.
\]
The density $h$ of $T$ is given by :
\[
h(x) = \left\{
    \begin{array}{ll}
        h_1=3/5  \ \text{if} & x \in I_1,\\
        h_2=6/5 \ \text{if} & x \in I_2,\\
        h_3=6/5 \ \text{if} & x \in I_3.\\
    \end{array}
\right.
\]
The DEI  $\theta_q$ can be easily computed  by equation (\ref{20}) and it reads:
\beq
\theta_q=
1-\frac{\frac{h_1^q}{(T_1')^{q-1}}+\frac{h_2^q}
{(T_2')^{q-1}}+\frac{h_3^q}{(T_3')^{q-1}}}{h_1^q+h_2^q+h_3^q}.
\nuq{26}

\item {\em Gauss map}. The Gauss map $T(x)= \frac1x-\left[\frac1x\right], x\in (0,1],$ has a.c. invariant density $h(x)=\frac1{\log 2(1+x)}$, which yields
\beq
\theta_q= 1- \frac{[\sum_{k=0,k\neq q-1}^{2(q-1)}(-1)^k {2(q-1) \choose k}\frac{2^{k-q+1}-1}{k-q+1}] + (-1)^{q-1} {2(q-1) \choose q-1}\log2}{\frac{2^{1-q}-1}{1-q}}.
\nuq{22}
Numerical computations for this map are reported in Figure \ref{fig-gauss}.

\begin{figure*}[h!]
    \centering
    \begin{subfigure}[t]{0.5\textwidth}
      \centering
       \includegraphics[height=2in]{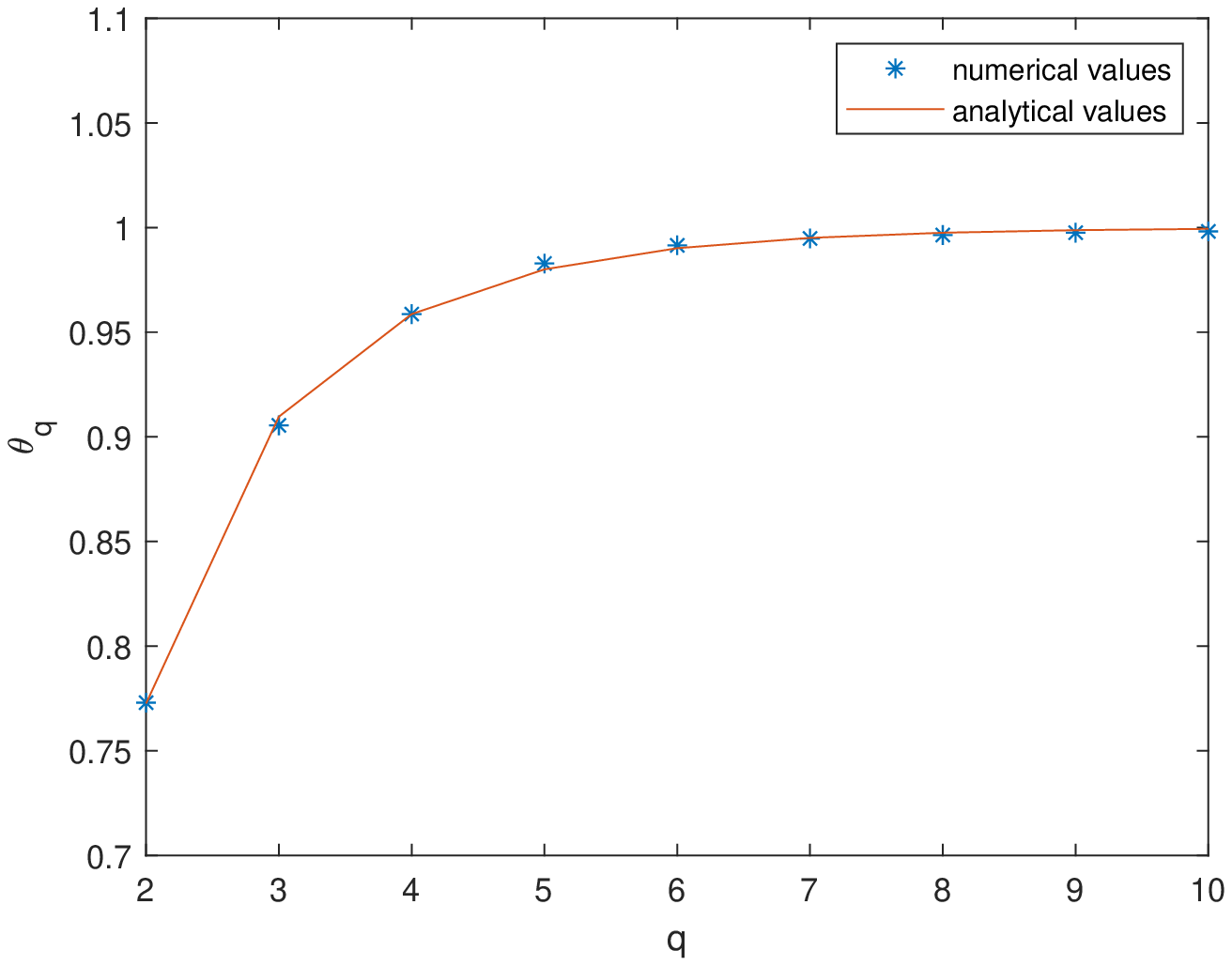}
        \caption{$\theta_q$}
    \end{subfigure}%
    ~
    \begin{subfigure}[t]{0.5\textwidth}
      \centering
       \includegraphics[height=2in]{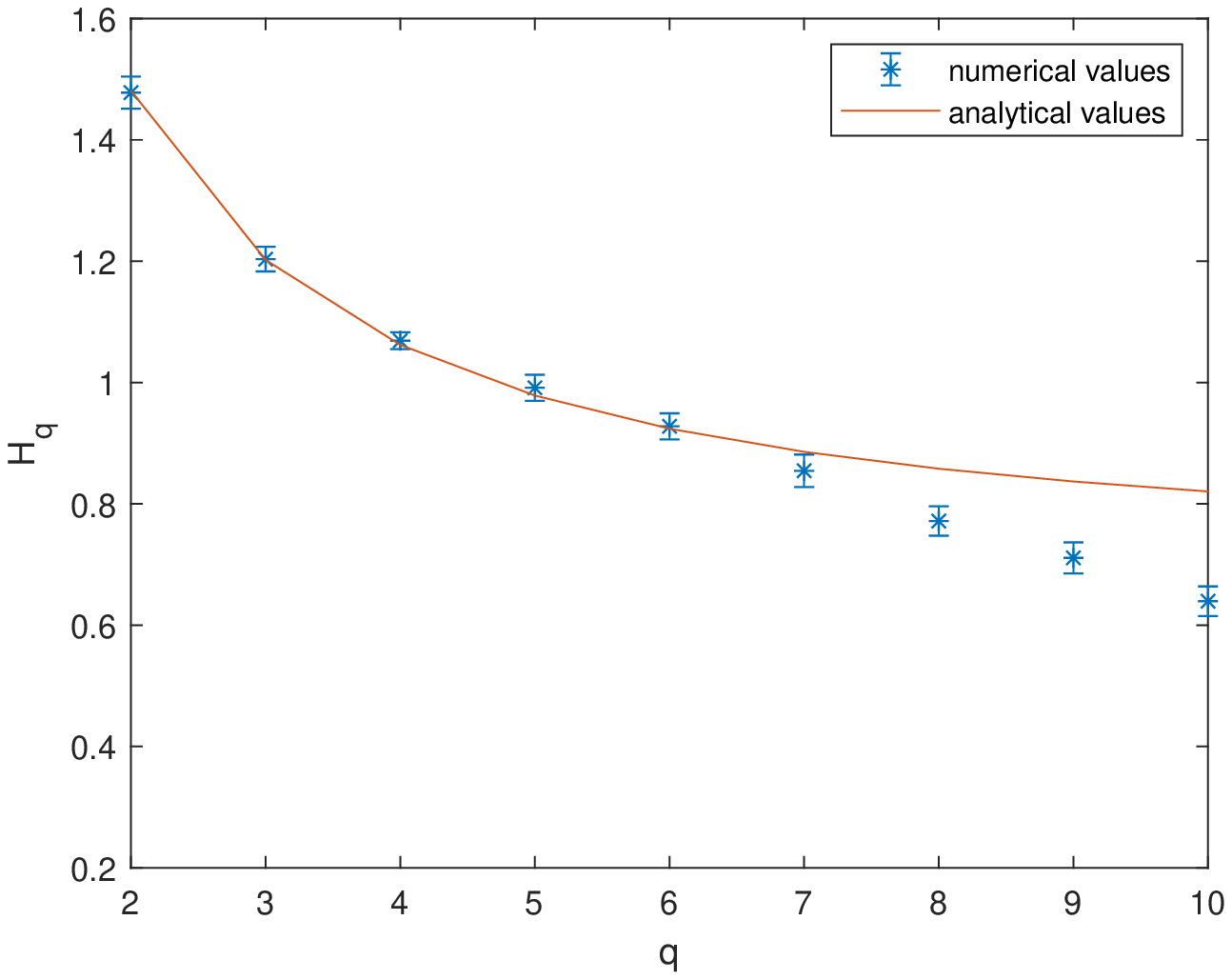}
        \caption{$H_q$}
    \end{subfigure}

    \caption{$\theta_q$ and $H_q$ of the Gauss map absolutely continuous invariant measure. Parameters of the numerical estimation are as in Fig. \ref{fig-3x}. }
\label{fig-gauss}\end{figure*}

\item {\em The Hemmer map.} This map \cite{hemmer} is defined on the interval $[-1,1]$ by  $T(x)= 1-2\sqrt{|x|}$ and it has the explicit density $h(x)=\frac{1}{2} (1-x)$ and Lyapunov exponent $1/2.$ The particularity of this density is that it vanishes in $x=1$.  The DEI for this case reads
\beq
 \theta_q= 1-\frac{q+1}{2^q} \sum_{k=0}^{q} {q \choose k}\frac{[1+(-1)^{k+q-1}]}{2k+q+1}.
\nuq{23}

\end{itemize}

\section{North Atlantic atmospheric variability}
\label{sec-climate}

In order to show the usefulness of our results in the study of many dimensional, complex systems, we compute the generalized dimensions associated with the atmospheric circulation over the North Atlantic. As observable, we consider the daily sea-level pressure fields observed in the region: [lat 22.5$^°$ N -- 70$^°$ N, lon 70$^°$ E -- 50$^°$ W] for the period 1948-2015, issued from the NCEP reanalysis  dataset \cite{NCEP}. Indeed, the sea-level pressure field is a proxy of the mid-latitude circulation as it traces the position of cyclones-anticyclones thanks to the rotation/stratification properties of atmospheric flows \cite{hoskin}.

References \cite{nature} and \cite{messori} computed the finite time local\footnote{Note that, in this context we also speak of daily dimension and persistence via EVT, meaning that each $z$ is the sea-level pressure field averaged during a day} dimensions $D_{1,r}(z)$ and persistence $\theta(z)$ of those fields with the method detailed in \cite{nature}, using as threshold the $98^{th}$ quantile of the observable distribution. It has been shown that those computations introduce a valuable piece of information in the description of the atmospheric flow at mid-latitudes. In particular i) the minima of the local dimension  $D_{1,r}(z)$ correspond to zonal flow circulation regimes, where the low pressure systems are confined to the polar regions, opposite to high pressure areas which insist on southern latitudes, in a North-South structure, ii) the maxima of the local dimension $D_{1,r}(z)$ correspond to blocked flows, where high and low pressure structures are distributed in Est-West direction. In this region, $D_{1,r}(z)$ takes values between 4 and 25, depending on the spatial resolution of the sea-level pressure field, with an average value of about 12.

In the previous sections we have described tools to compute the generalized dimensions and to link them to local dimensions. In principle, the method described in section \ref{sub-local} requires the computation of a distribution of local dimensions at a uniform resolution $r_{\mbox{\tiny cut}}$. The nature of data at our disposal is not suited to an analysis at fixed resolution, since it may oversample, or undersample the distributions of extreme events of the observable $\phi_z(x)=-\log(d(z,x))$, depending on the point $z$ that we are considering. For this reason, we follow the approach described in \cite{nature} and \cite{messori} and use as threshold values $T_{p,z}$ for the observable $\phi_z$ the $p$-quantile of the distribution of $\phi_z$, where $p$ is fixed. This ensures that the extreme value statistics is computed with the same sample statistics at all points. The effective radius considered for the computation of the generalized dimensions is then taken to be the average of $e^{-T_{p,z}}$ over $z$. Applying formula (\ref{8}) to the computation of $D_q$, one obtains the non-linear behavior pictured in Fig. \ref{climateDq}. We give a summary of the results found when adopting the above procedure with different quantiles in Table \ref{table}.  When the quantile is relatively low, a large sample of recurrences is used (corresponding to a larger average cutoff radius). This implies a lower spread of the distributions of $D_{1,r}(z)$. To the contrary, when the quantile is larger, the sample statistics contains fewer recurrences and the spread in $D_{1,r}(z)$ increases. Note that, although $\min(D_{1,r}(z))$ and $\max(D_{1,r}(z))$ seem to experience large variations with different quantiles, these have to be compared with the dimension of the phase space, which corresponds here to the number of grid points of the sea-level pressure fields used, 1060.  The relative variation is therefore very small, less than 1\%. Depending on the size of the datasets, one can then look for the best estimates of the $D_{1,r}(z)$ distribution for several values of $q$ and look for a range of stable estimates in $q$ space.

\begin{tabular}{|c|c|c|c|c|}
  \hline
  $p$ &  $D_{-\infty}=\min(D_{1,r}(z))$ & $D_1=\overline{D_{1,r}(z)}$ & $D_2$ & $D_\infty=\max(D_{1,r}(z))$  \\
  \hline
  0.95 & 20.7 & 11.2 & 9.5 & 6.0\\
  0.97 & 24.2 &12.2 & 10.2 &6.4 \\
  0.98 & 25.7 & 13.0 & 10.5 & 6.4\\
  0.99 & 29.1 & 14.3 & 11.1 & 6.5 \\
  0.995 & 39.6 & 15.5 & 11.3 & 6.2\\
  \hline

\end{tabular}
\captionof{table}{Values of $D_q$ found with different quantiles. For all of them, estimates $D_{-\infty}$ and $D_\infty$  match with the extrema of the local dimensions, and $D_1$ with the average of the local dimensions. Estimates of $D_2$ are larger than the value of 8.9 found with a different technique in \cite{D2}.}
 \label{table}

\begin{figure}[h!]
\includegraphics[width=.7\textwidth]{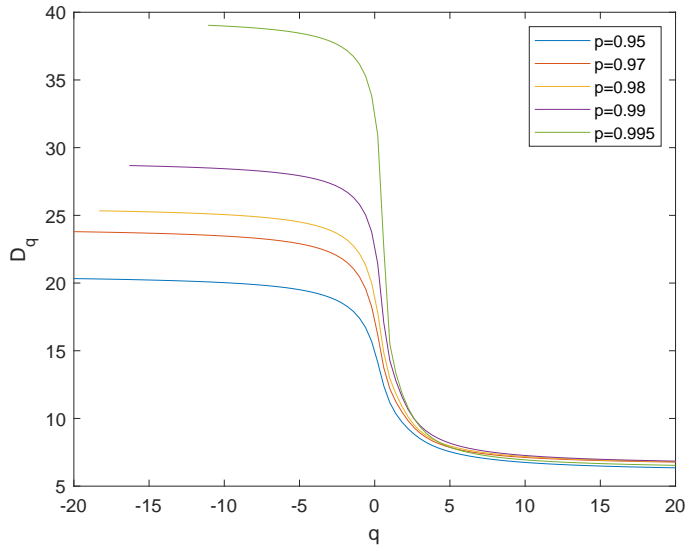}
\caption{$D_q$ spectrum obtained from climate data, using equation (\ref{8}) and the techniques described in the text. Curves are displayed for different values of the quantile $p$.}
\label{climateDq}
\end{figure}

                                                                                                                                                                                                                                                \vspace{1\baselineskip}
Generalized dimensions are a piece of information that is typical of the invariant, \textit{ultimate} measure, that is, the mathematical object that ergodic theory defines in the infinite time limit; nonetheless, as shown in this paper and in the papers quoted in the references, different techniques exist to extrapolate their value from data generated from observations which are not yet asymptotic. One could call the corresponding objects   {\em penultimate}, in analogy with extreme value statistics where the adjective {\em penultimate} is used to describe the probability distribution of extremes of a finite size sample. Indeed, while it is well known for a large class of systems \cite{LSY} that with probability one the local dimensions coincide with the information dimension, at finite resolution large deviation theory estimates the likelihood of deviations from this value. In this perspective, the spread of the experimentally observed values of $D_{1,r}(z)$ can be thought of as originating from the multifractal structure of the ultimate invariant measure, which in turn is revealed by the non-constant value of generalized dimensions.

\section{Discussion and Perspectives}

In this paper we have explored the relations between the spectrum of generalized dimensions $D_q$ and the recurrence properties of the dynamics. In fact, the former determines the large deviations of dynamical quantities such as return times \cite{SJ} and hitting times: \cite{CU} and Proposition \ref{prop-two} herein. The statistics of hitting times ruled by Proposition \ref{prop-1} also opens the way to new techniques to estimate generalized dimensions via recurrence properties. We have also seen that many of these concepts can be given a fruitful interpretation within extreme events theory, with a significant potential for application to experimental data.

The relation between extreme value theory and large deviations in the context of recurrence is a promising new field of research that we plan to extend to concrete situations in natural sciences, like climate, turbulence, and neural networks. The climate dynamics data shown in this paper are a first example of this endeavor. Here, atmospheric extreme events (like {\em e.g.} extratropical storms or blocking) produce large excursions of the local dimensions $D_{1,r}(z)$, which in turn are associated with large deviations of hitting and return times, in the proximity of special points in phase space.  Since the computation of local dimensions is relatively feasible also for systems with a high number of degrees of freedom, this can be used to trace the location of singularities originating the multifractal $D_q$ spectra.

\section{Aknowledgements}
Intensive numerical computations for this paper have been performed on the INFN cluster at the University of Pisa, Italy. \\
S. V. was supported  by the MATH AM-Sud Project
Physeco and by the project APEX Syst\'emes dynamiques: Probabilit\'es et Approximation
Diophantienne PAD funded by the R\'egion PACA (France). He also warmly thanks the LabEx Archimede (AMU University, Marseille), INdAM (Italy) for additional support and J-R. Chazottes and B. Saussol for interesting discussions related to this paper.
P.Y. and D.F. were supported by ERC grant no. 338965 (A2C2).

\end{document}